\documentclass[11pt,twoside]{amsart}
\usepackage{amsmath,amsfonts,amssymb,mathrsfs}
\usepackage[english]{babel}
\usepackage[utf8]{inputenc}
\usepackage[T1]{fontenc}
\usepackage{pgf,tikz,pgfplots}
\pgfplotsset{compat=1.15}
\usetikzlibrary{babel}
\usepackage{mathrsfs}
\usepackage{bbold}
\usepackage{enumitem}
\usetikzlibrary{arrows}
\usepackage[left=1.8cm,right=1.8cm,top=2cm,bottom=2cm]{geometry}
\newtheorem{theorem}{Theorem}[section]

\newtheorem{lemma}[theorem]{Lemma}

\newtheorem{prop}[theorem]{Proposition}

\newtheorem{remark}{Remark}[section]

\renewcommand{\Re}{\operatorname{Re}}

\def\R{\mathbb R}

\def\N{\mathbb N}
\def\Z{\mathbb Z}

\numberwithin{equation}{section}

\usepackage{hyperref}

\DeclareMathOperator{\dive}{div}

\usepackage{tikz}
\usetikzlibrary{decorations}
\usetikzlibrary{decorations.pathmorphing}
\usetikzlibrary{decorations.pathreplacing}
\usetikzlibrary{decorations.shapes}
\usetikzlibrary{decorations.text}
\usetikzlibrary{decorations.markings}
\usetikzlibrary{decorations.fractals}
\usetikzlibrary{decorations.footprints}
\title{Parabolic-elliptic Keller-Segel's system}
\author[Valentin Lemarié]{Valentin Lemarié}
\begin{document}
 {\begin{center}
\begin{abstract} 
We study on the whole space $\R^d$ the compressible Euler system with damping coupled to the Poisson equation when the damping coefficient tends towards infinity. We first prove a result of global existence for the Euler-Poisson system in the case where the damping is large enough, then, in a second step, we rigorously justify  the passage to the limit
to the parabolic-elliptic Keller-Segel after performing a diffusive rescaling, and get an explicit convergence rate.
The overall study is 
carried out in `critical' Besov spaces, in the spirit of the recent survey \cite{RD} by 
 R. Danchin devoted to partially dissipative systems. 
 \end{abstract}\end{center}}
\maketitle

\section{Introduction}
In this article, we will focus on two systems : Euler-Poisson and parabolic-elliptic Keller-Segel system. Let us first
present  these systems and motivate their study.

\medbreak
The Euler-Poisson system with damping, set on the whole space $\R^d$ (where $d\geq 2$) reads: 
\begin{eqnarray}\label{Euler-Poisson initial}\left\{ \begin{array}{ll} \partial_t \rho^{\varepsilon}+\dive(\rho^{\varepsilon} v^{\varepsilon})=0, \\ \partial_t(\rho^{\varepsilon} v^{\varepsilon}) + \dive(\rho^{\varepsilon} v^{\varepsilon} \otimes v^{\varepsilon})+\nabla \left( P \left( \rho^{\varepsilon}\right) \right)+\varepsilon^{-1}\rho^{\varepsilon} v^{\varepsilon} = - \rho^{\varepsilon} \nabla V^\varepsilon, \\ - \Delta V^\varepsilon= \rho^{\varepsilon}-\overline{\rho} \end{array} \right. \end{eqnarray}
where $\varepsilon>0$, $\rho^\varepsilon=\rho^\varepsilon(t,x)\in\R_+$ represents the density of the gas (with $\overline{\rho}>0$ a constant state), $v^\varepsilon=v^\varepsilon(t,x)\in\R^d$ the velocity, the pressure $P(z)=Az^\gamma$ with $\gamma>1$ and $A>0$, as well as $V^\varepsilon=V^\varepsilon(t,x)$ the potential.

This is the classical Euler compressible system with  damping, to which we added a coupling with the Poisson equation.

Without this coupling, System \eqref{Euler-Poisson initial} reduces to  the  compressible Euler system with the damping coefficient~$\varepsilon^{-1}$ : 
$$\left\{ \begin{array}{ll} \partial_t \rho^{\varepsilon}+\dive(\rho^{\varepsilon} v^{\varepsilon})=0, \\ \partial_t(\rho^{\varepsilon} v^{\varepsilon}) + \dive(\rho^{\varepsilon} v^{\varepsilon} \otimes v^{\varepsilon})+\nabla \left( P \left( \rho^{\varepsilon}\right) \right)+\varepsilon^{-1}\rho^{\varepsilon} v^{\varepsilon} = 0. \end{array} \right. $$

This system was recently studied by Crin-Barat and Danchin in \cite{CBD2} where they established a result of existence and uniqueness of global solutions for sufficiently small data, and obtained optimal time decay estimates for these solutions. In parallel, they  studied the singular limit of the system when the damping coefficient tends towards infinity. They then obtained the equation of porous media. In this article, we will draw freely on this study and the method used to obtain a priori estimates. 

The system that we will study is usually used to describe the transport of charge carriers (electrons and ions) in semiconductor devices or plasmas.  The system consists of conservation laws for mass density and current density for carriers, with a Poisson equation for electrostatic potential. This system is also of interest in other fields like e.g. in chemotaxis:
then, $\rho^\varepsilon$ represents cell density and $V^\varepsilon,$ the concentration of chemoattractant secreted by cells. 
For recent results on the Euler-Poisson system: well-posedness, existence of global solutions, study of long-time behavior or other singular limits, the reader can refer to \cite{EulerP1}, \cite{EulerP2}, \cite{EulerP3}, \cite{EulerP4}, \cite{EulerP5},\cite{EulerP6}, \cite{EulerP7}, \cite{EulerP8} and \cite{EulerP9}.

As in \cite{RD}, in order to investigate the asymptotics of solutions of \eqref{Euler-Poisson initial}
when $\varepsilon$ goes to $0,$ we perform 
the following so-called \emph{diffusive} change of  variable:  \begin{eqnarray}\label{changement de variable diffusif}(\tilde{\varrho}^\varepsilon,\tilde{v}^\varepsilon)(\tau,x)=(\varrho^\varepsilon,\varepsilon^{-1} v^\varepsilon)(\varepsilon^{-1} \tau,x)\end{eqnarray}
so that we have
\begin{eqnarray}\label{système avant passage à la limite}\left\{ \begin{array}{l}
\displaystyle \partial_t \tilde{\varrho}^\varepsilon+\dive(\tilde{\varrho}^\varepsilon \tilde{v}^\varepsilon)=0  \\ \displaystyle \varepsilon^2 \left(\partial_t \tilde{v}^\varepsilon+\tilde{v}^\varepsilon\cdot \nabla \tilde{v}^\varepsilon\right)+\frac{\nabla\left(P(\tilde{\varrho}^\varepsilon)\right)}{\tilde{\varrho}^\varepsilon}+\tilde{v}^\varepsilon+\nabla(-\Delta)^{-1}(\tilde{\varrho}^\varepsilon-\overline{\varrho}) =0 .
\end{array}\right.\end{eqnarray}
We then define the damped mode: 
\begin{eqnarray}\label{def mode amorti} \displaystyle\tilde{W}^\varepsilon\mathrel{\mathop:}=\frac{\nabla\left(P(\tilde{\varrho}^\varepsilon)\right)}{\tilde{\varrho}^\varepsilon}+\tilde{v}^\varepsilon+\nabla(-\Delta)^{-1}(\tilde{\varrho}^\varepsilon-\overline{\varrho}).\end{eqnarray}

As the first equation of \eqref{système avant passage à la limite} can be rewritten as $$\partial_t \tilde{\varrho}^\varepsilon-\Delta\left(P(\tilde{\varrho}^\varepsilon)\right)-\dive\left(\tilde{\varrho}^\varepsilon \nabla (-\Delta)^{-1}(\tilde{\varrho}^\varepsilon-\overline{\varrho})\right)=\dive(\tilde{\varrho}^\varepsilon \tilde{W}^\varepsilon),$$ we expect the limit density  $N$ to satisfy the following \emph{parabolic-elliptic Keller-Segel system} : \begin{eqnarray}\label{Keller-Segel} \left\{ \begin{array}{l}\partial_t N-\Delta\left(P(N)\right)=\dive\left(N \ \nabla V\right) \\ -\Delta V=N-\overline{\varrho} \end{array} \right. \end{eqnarray} supplemented with the initial data $\displaystyle \underset{\varepsilon\to 0}{\lim}\,\tilde{\varrho}_0^\varepsilon$.

\medbreak
Our second aim is to justify the passage to the limit when $\varepsilon\to 0$ of the Euler-Poisson system towards the parabolic-elliptic Keller-Segel system. 

Recall that \eqref{Keller-Segel} is a model for 
 describing  the evolution of density $N=N(t,x)\in \R_+$ of a biological population under the influence of a chemical agent with concentration $V=V(t,x)\in \R^d$. Chemotaxis are an important means of cell communication. How cells are arranged and organized is determined by communication by chemical signals. Studying such a biological process is important because it has repercussions in many branches of medicine such as cancer \cite{cancer1}, \cite{cancer2}, embryonic development  \cite{embryon} or vascular networks \cite{vasc1}, \cite{vasc2}.
The previous system is famous in biology and comes from E.F Keller and L.A Segel in \cite{keller-segel1}. This basic model was used to describe the collective movement of bacteria possibly leading to cell aggregation by chemotactic effect. We refer to the articles \cite{keller-segel2} and \cite{keller-segelall} for more details and information about the different Keller-Segel models studied since the 1970s. 

Our aim here is to demonstrate that \eqref{Keller-Segel} may be obtained from the Euler-Poisson system with damping when the  parameter $\varepsilon$ tends to $0.$ 
This question has been addressed   in \cite{Euler-Poisson tore} on the torus case and Sobolev spaces
 in a situation  where the potential satisfies a less singular equation : the author justifies the passage to the limit for regular periodic solutions. 
A lot of articles justify another limit:
the passage from the parabolic-parabolic Keller-Segel system to the  parabolic-elliptic Keller-Segel system (see e.g. the paper
\cite{Lemarié PG} by P-G. Lemarié-Rieusset for the case
of  Morrey spaces).

In the same spirit as this article, T. Crin-Barat, Q. He and L. Shou 
in \cite{CBHS} justified the  high relaxation asymptotics  
for the (less singular) parabolic-parabolic Keller-Segel system (the potential satisfies the equation $-\Delta V+bV=aN$ with $a,b>0$) : this other system comes from the system (HPC) (hyperbolic-parabolic-chemotaxis) which is a damped isentropic compressible Euler system  with a potential satisfying an elliptical equation. In comparison with what is done here, 
T. Crin-Barat \emph{et al} used a parabolic approach to 
justify their passage to the limit. Here,  we have to handle 
the more singular case where the limit system is parabolic-elliptic. 

\section{Main results and sketch of the proof}
In this section, we will first present and motivate the functional spaces used. Secondly we will state the results and the sketch of the proofs about the well-posedness behavior of Euler-Poisson system and the justification of the passage to the limit to parabolic-elliptic Keller-Segel system.
\subsection{Functional spaces}~\\
Before describing the main results of this article, we introduce the different notations and definitions used throughout this document. 
We will designate by $C>0$ an independent constant of $\varepsilon$ and time, and $f\lesssim g$ will mean $f\leq C g$. For any Banach space $X$ and all functions $f,g\in X$, we denote $\|(f,g)\|_X\mathrel{\mathop:=}\|f\|_X+\|g\|_X$. We designate by $L^2(\R_+;X)$ the set of measurable functions $f:[0,+\infty[\rightarrow X$ such that $t\mapsto \|f(t)\|_{X}$ belongs to $L^2(\R_+)$ and write $\|\cdot \|_{L^2(\R_+;X)}\mathrel{\mathop:}=\|\cdot \|_{L^2(X)}$.

In this article we will use a decomposition in Fourier space, called the \emph{homogeneous Littlewood-Paley decomposition}. 
To this end, we introduce  a regular non-negative function $\varphi$ on $\R^d$ with support in the annulus $\{\xi\in\R^d,\:  3/4\leq|\xi|\leq 8/3\}$ and satisfying $$\sum_{j\in\Z}\varphi(2^{-j}\xi)=1,\qquad \xi\not=0.$$ 
For all $j\in\Z$, the dyadic homogeneous blocks $\dot \Delta_j$ and the low frequency truncation operator $\dot S_j$ are defined by $$\dot \Delta_j \mathrel{\mathop:}=\mathcal{F}^{-1}(\phi(2^{-j}\cdot)\mathcal{F}u), \quad \dot S_j u\mathrel{\mathop:}=\mathcal{F}^{-1}(\chi(2^{-j}\cdot)\mathcal{F}u),$$ where $\mathcal{F}$ and $\mathcal{F}^{-1}$ designate respectively the Fourier transform and its inverse. From now on, we will use the following shorter notation : $$ u_j:=\dot\Delta_j u.$$

Let $\mathcal{S}_h'$ the set of tempered distributions $u$ on $\R^d$ such that $\displaystyle \underset{j\to -\infty}{\lim}\|\dot S_j u\|_{L^\infty}=0$. We have then : $$u=\sum_{j\in \Z} u_j \ \in \mathcal{S}', \quad \dot S_j u=\sum_{j'\leq j-1} u_{j'}, \ \forall u\in \mathcal{S}_h'.$$
Homogeneous Besov spaces $\dot B_{p,r}^s$ for all $p,r\in[1,+\infty]$ and $s\in\R$ are defined by: $$\dot B_{p,r}^s\mathrel{\mathop:}=\left\{u\in \mathcal{S}_h' \middle| \|u\|_{\dot B_{p,r}^s}\mathrel{\mathop:}=\|\{2^{js}\|u_j\|_{L^p}\}_{j\in\Z}\|_{l^r}<\infty\right\}\cdotp$$
In this article, we will only consider  Besov spaces of indices $p=2$ and $r=1$. 

As we will need to restrict our Besov norms in specific regions of low and high frequencies, we introduce the following notations : \begin{center}
$\displaystyle\|u\|_{\dot B_{2,1}^s}^h\mathrel{\mathop:}=\sum_{j\geq -1}2^{js}\|u_j\|_{L^2}$, $\displaystyle \ \|u\|_{\dot B_{2,1}^s}^l\mathrel{\mathop:}=\sum_{j\leq -1}2^{js}\|u_j\|_{L^2}$ ,  \ $ \displaystyle\|u\|_{\dot B_{2,1}^s}^{l^-,\varepsilon}\mathrel{\mathop:}=\sum\limits_{\underset{2^j\leq \varepsilon}{j\leq -1}}2^{js}\|u_j\|_{L^2}$, \  $\displaystyle \|u\|_{\dot B_{2,1}^s}^{l^+,\varepsilon}\mathrel{\mathop:}=\sum\limits_{\underset{2^j\geq \varepsilon}{j\leq -1}}2^{js}\|u_j\|_{L^2}.$
\end{center}
We put in the appendix several results about Besov spaces: the reader may refer to Chapter 2 of \cite{BCD} for more information on this topic.

\subsection{Main results, sketch of the proofs and article organization}~\\
The starting point is that, formally, \eqref{Euler-Poisson initial} rewrites :

\begin{eqnarray}\label{système initial}\left\{ \begin{array}{ll} \partial_t \rho^{\varepsilon}+\dive(\rho^{\varepsilon} v^{\varepsilon})=0, \\ \partial_t(\rho^{\varepsilon} v^{\varepsilon}) + \dive(\rho^{\varepsilon} v^{\varepsilon} \otimes v^{\varepsilon})+\nabla \left( P \left( \rho^{\varepsilon}\right) \right)+\varepsilon^{-1}\rho^{\varepsilon} v^{\varepsilon} =-\rho^{\varepsilon} \nabla \left(-\Delta\right)^{-1}\left(\rho^{\varepsilon}-\overline{\rho}\right). \end{array} \right. \end{eqnarray}

By the variable change \begin{eqnarray}\label{changement de variable 2}
    (\rho,v)=(\rho^{\varepsilon},v^{\varepsilon})(\varepsilon t, \varepsilon x),\end{eqnarray} we get the following system:  : 
\begin{eqnarray}\label{système 2}\left\{ \begin{array}{ll} \partial_t \rho+\dive(\rho v)=0, \\ \partial_t(\rho v) + \dive(\rho v \otimes v)+\nabla \left( P \left( \rho\right) \right)+\rho v =- \varepsilon^2  \rho \nabla \left(-\Delta\right)^{-1}\left(\rho-\overline{\rho}\right). \end{array} \right. \end{eqnarray}

For the study of this system, the key is to obtain suitable global-in-time a priori estimates. 
Then,  very classical  arguments lead to existence and uniqueness of global solutions (see Theorem  \ref{théorème Euler-Poisson}
below).

Obtaining estimates will be strongly inspired by the work done by Crin-Barat \emph{et al} in \cite{CBD2} where we consider the classic compressible Euler system. As the system we are studying is very close to a partially dissipative system, we
will follow  \cite{RD} so as to obtain a priori estimates : the standard energy method is not enough to conclude because we do not obtain all the information through this (mainly at the low frequencies) on the dissipated part. Hence, we must  better use the coupling,  exhibiting a combination of unknowns (the "purely" damped mode) that will allow us to recover the
whole dissipation. 

For the estimates, therefore, we follow the strategy proposed by Danchin in \cite{RD}: 
first (second part of the third section), we analyze how to obtain the estimates for the linear system in Besov space $\dot B_{2,1}^{s}$ where $s\in\R$ is any at the moment. For high frequencies, we follow step by step the approach of \cite{RD}, by putting the negligible term containing $\varepsilon^2 \nabla \Delta^{-1} \varrho$,  in order to obtain exponential decay.
For low frequencies, the task is slightly more complicated because the latter term is no more negligible : we lose the symmetry condition and we can not apply Danchin's method. 
But by looking at the system differently (essentially by changing $v$ to another variable dependent on $ \varepsilon$), we obtain a symmetrical system for which we find the associated estimates (go back to the initial unknowns to get the estimates). We then see the appearance of two regimes within the low frequencies that will be named respectively very low frequencies (frequencies below $ \varepsilon$) and medium frequencies (frequencies of magnitude between $ \varepsilon$ and $ 1$). 
To recover the full dissipative properties of the system, we introduce  the \emph{damped mode} $W\mathrel{\mathop:}=-\partial_t v.$  Then, from the  estimate satisfied by $W,$  we will improve the estimate for the low frequency part of $v.$ 

Let us next explain how  to choose the indices of regularity for the solution. Since our system is very similar to Euler’s without coupling, we make the same choice as in \cite{CBD2}~: $s=\frac{d}{2}$ for the medium frequencies and $s=\frac{d}{2}+1$ for the high frequencies. We take  $\frac{d}{2}-1$ for the very low frequencies of the density in view of the estimate obtained for the linear system.

For proving a priori estimates for \eqref{système 2}, we have to take into account non-linear terms now. To do this, for high frequencies, we again follow the method described by \cite{RD} with a precise analysis of the system using commutators.

Concerning the low frequencies, we need some information on the damped mode $W=-\partial_t v$, before studying the estimates of the density and velocity. By combining the obtained inequalities,we deduce the desired a priori estimates.

Before stating our main  results, 
we provide the reader with the following diagram so as to clarify the notations of the theorem:

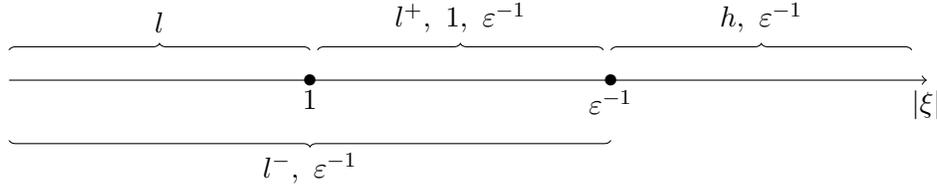
\begin{figure}[!h]
\begin{center}
	\begin{tikzpicture}{width=15cm}

\begin{scope}[scale = 2]



\draw[->] (-2,0) -- (4.1,0);
\node[below] (0,0) {$1$};
 \draw (2,0) node{$\bullet$};
 \draw (2,0) node[below]{$\varepsilon^{-1}$};
 \node (2,0) {$\bullet$};

\draw (4.1,0) node[below]{$|\xi|$};
\draw[color=black,decorate,decoration={brace,raise=0.2cm}]
(-2,0.1) -- (0,0.1) node[above=0.3cm,pos=0.5] {$l$};
\draw[color=black,decorate,decoration={brace,raise=0.2cm}]
(0.05,0.1) -- (1.95,0.1) node[above=0.3cm,pos=0.5] {$l^+, \ 1 , \ \varepsilon^{-1}$};
\draw[color=black,decorate,decoration={brace,raise=0.2cm}]
(2,0.1) -- (4,0.1) node[above=0.3cm,pos=0.5] {$h, \ \varepsilon^{-1}$};
\draw[color=black,decorate,decoration={brace,raise=0.2cm}]
(2,-0.3) -- (-2,-0.3) node[above=-0.9cm,pos=0.5] {$l^-,\ \varepsilon^{-1}$};

\end{scope}	
\end{tikzpicture}
\end{center}	
\end{figure}


Our first result states the global well-posedness  of the Euler Poisson system with high relaxation. We point out an explicit dependence of the estimates with respect to the damping parameter, that we believe to be optimal:

\begin{theorem}\label{théorème Euler-Poisson} Let $\varepsilon>0$.
There exists a positive constante $\alpha$ such that for all $\varepsilon\leq {1}/{2}$ and initial data $Z_0^\varepsilon=(\varrho_0^\varepsilon-\overline{\varrho},v_0^\varepsilon)\in \left(\dot B_{2,1}^{\frac{d}{2}-1}\cap \dot B_{2,1}^{\frac{d}{2}+1}\right)\times \left(\dot B_{2,1}^{\frac{d}{2}}\cap \dot B_{2,1}^{\frac{d}{2}+1}\right) $ satisfying : 

$$\mathcal{Z}_0^{\varepsilon}\mathrel{\mathop:}= \|\varrho_0^\varepsilon-\overline{\varrho}\|_{\dot B_{2,1}^{\frac{d}{2}-1}}^{l}+ \|\varrho_0^\varepsilon-\overline{\varrho}\|_{\dot B_{2,1}^{\frac{d}{2}}}^{l^+,\ 1, \ \varepsilon^{-1}}+\|v_0^\varepsilon\|_{\dot B_{2,1}^{\frac{d}{2}}}^{l^{-},\ \varepsilon^{-1}}+\varepsilon\|\left(\varrho_0^\varepsilon-\overline{\varrho},v_0^\varepsilon\right)\|_{\dot B_{2,1}^{\frac{d}{2}+1}}^{h,\ \varepsilon^{-1}}\leq \alpha, $$ System \eqref{Euler-Poisson initial} 
supplemented with the initial data $(\varrho_0^\varepsilon-\bar\rho,v_0^\varepsilon)$  admits a unique global-in-time solution $Z^\varepsilon=(\varrho^\varepsilon-\bar\rho,v^\varepsilon)$ in the set 

\begin{multline*}
E\mathrel{\mathop:}=\bigg\{(\varrho^\varepsilon-\overline{\varrho},v^\varepsilon) \ \bigg| \  (\varrho^\varepsilon-\overline{\varrho})^{l}\in \mathcal{C}_b(\R_+:\dot B_{2,1}^{\frac{d}{2}-1}), \ \varepsilon(\varrho^\varepsilon-\overline{\varrho})^{l}\in L^1(\R_+; \dot B_{2,1}^{\frac{d}{2}-1}), \\ \ (\varrho^\varepsilon-\overline{\varrho})^{l^+, \ 1, \ \varepsilon^{-1}}\in \mathcal{C}_b(\R_+:\dot B_{2,1}^{\frac{d}{2}}), \ 
 \varepsilon(\varrho^\varepsilon-\overline{\varrho})^{l^+, \ 1, \ \varepsilon^{-1}}\in L^1(\R_+;\dot B_{2,1}^{\frac{d}{2}+1}), \\  (v^\varepsilon) ^{l^-, \ \varepsilon^{-1}}\in \mathcal{C}_b(\R_+;\dot B_{2,1}^{\frac{d}{2}}), \ (v^\varepsilon)^{l}\in L^1(\R_+;\dot B_{2,1}^{\frac{d}{2}}), \ (v^\varepsilon)^{l^+,\ 1, \ \varepsilon^{-1}}\in L^1(\R_+;\dot B_{2,1}^{\frac{d}{2}+1}), \\ (\varrho^\varepsilon-\overline{\varrho},v^\varepsilon)^{h, \varepsilon^{-1}}\in\mathcal{C}_b(\R_+;\dot B_{2,1}^{\frac{d}{2}-1})\cap L^1(\R_+;\dot B_{2,1}^{\frac{d}{2}+1}), \ w^\varepsilon\in \mathcal{C}_b(\R_+;\dot B_{2,1}^{\frac{d}{2}})\cap L^1(\R_+;\dot B_{2,1}^{\frac{d}{2}}) \bigg\}
\end{multline*}

where we have denoted $\displaystyle w^\varepsilon\mathrel{\mathop:}=\varepsilon\frac{\nabla\left(P(\varrho^\varepsilon)\right)}{\varrho^\varepsilon}+ v^\varepsilon+\varepsilon\nabla (-\Delta)^{-1}(\varrho^\varepsilon-\overline{\varrho})$.

~

Moreover, we have the following inequality : $$\mathcal{Z}^\varepsilon(t)\leq C \mathcal{Z}_0^\varepsilon$$ where \begin{multline*}
    \mathcal{Z}^\varepsilon(t)\mathrel{\mathop:}= \displaystyle\|\varrho^\varepsilon-\overline{\varrho}\|_{L^\infty\big(\dot B_{2,1}^{\frac{d}{2}-1}\big)}^{l}+\|\varrho^\varepsilon-\overline{\varrho}\|_{L^\infty\big(\dot B_{2,1}^{\frac{d}{2}}\big)}^{l^+,\ 1, \ \varepsilon^{-1} }+\|v^\varepsilon\|_{L^\infty \big(\dot B_{2,1}^{\frac{d}{2}}\big)}^{l^{-}, \  \varepsilon^{-1}}+\varepsilon\|(\varrho^\varepsilon-\overline{\varrho},v)\|_{L^\infty\big(\dot B_{2,1}^{\frac{d}{2}+1}\big)}^{h, \ \varepsilon^{-1}}
    \\
   + \|\varepsilon(\varrho^\varepsilon-\overline{\varrho})\|_{L^1\big(\dot B_{2,1}^{\frac{d}{2}-1}\big)}^{l}+\|v^\varepsilon\|_{L^1\big(\dot B_{2,1}^{\frac{d}{2}}\big)}^{l}+\varepsilon\|\varrho^\varepsilon-\overline{\varrho}\|_{L^1\big(\dot B_{2,1}^{\frac{d}{2}+2}\big)}^{l^+, \ 1 , \ \varepsilon^{-1}}+\|v^\varepsilon\|_{L^1\big(\dot B_{2,1}^{\frac{d}{2}+1}\big)}^{l^+,\ 1, \  \varepsilon^{-1}}\\ +\|(\varrho^\varepsilon-\overline{\varrho},v)\|_{L^1\big(\dot B_{2,1}^{\frac{d}{2}+1}\big)}^{h, \varepsilon^{-1}} +\|w^\varepsilon\|_{L^\infty\big(\dot B_{2,1}^{\frac{d}{2}}\big)}+\varepsilon^{-1}\|w^\varepsilon\|_{L^1\big(\dot B_{2,1}^{\frac{d}{2}}\big)}.
\end{multline*}

\end{theorem}
\begin{remark}
In summary, for large enough damping and small enough initial data,  we get a global solution to the Euler-Poisson system. The demonstration will allow us to understand the choice of regularity indices for the different frequency groups. In addition, 
the control of the  damped mode $w^\varepsilon$ in the above theorem which enable us to obtain on the one hand the a priori estimate of the theorem and on the other hand the following theorem on the singular limit of \eqref{Euler-Poisson initial}  when the damping coefficient $\varepsilon^{-1}$ tends to infinity.
\end{remark}

By the change of variable \eqref{changement de variable diffusif} and the existence theorem on Euler-Poisson, we then have $\displaystyle\tilde{W}^\varepsilon=\mathcal{O}(\varepsilon)$ in $L^1(\R_+;\dot B_{2,1}^{\frac{d}{2}})$ where $\displaystyle \tilde{W}^\varepsilon$ is defined by \eqref{def mode amorti} and we need to look at the solutions of \eqref{système avant passage à la limite} in the following functional space : 
\begin{multline*}
\tilde{E}\mathrel{\mathop:}=\bigg\{(\varrho^\varepsilon-\overline{\varrho},v^\varepsilon)  \ \bigg| \  (\varrho^\varepsilon-\overline{\varrho})^{l}\in \mathcal{C}_b(\R_+:\dot B_{2,1}^{\frac{d}{2}-1}), \ \varepsilon(\varrho^\varepsilon-\overline{\varrho})^{l}\in L^1(\R_+; \dot B_{2,1}^{\frac{d}{2}-1}), \\ \ (\varrho^\varepsilon-\overline{\varrho})^{l^+, \ 1, \ \varepsilon^{-1}}\in \mathcal{C}_b(\R_+:\dot B_{2,1}^{\frac{d}{2}}), \ 
 \varepsilon(\varrho^\varepsilon-\overline{\varrho})^{l}\in L^1(\R_+;\dot B_{2,1}^{\frac{d}{2}+1}), \ (\varrho^\varepsilon-\overline{\varrho})^{l^+, \ 1, \ \varepsilon^{-1}}\in L^1(\R_+;\dot B_{2,1}^{\frac{d}{2}+1}), \\  \varepsilon(v^\varepsilon) ^{l-, \ \varepsilon^{-1}}\in \mathcal{C}_b(\R_+;\dot B_{2,1}^{\frac{d}{2}}), \ (v^\varepsilon)^{l}\in L^1(\R_+;\dot B_{2,1}^{\frac{d}{2}}), \ (v^\varepsilon)^{l^+, \ 1, \ \varepsilon^{-1}}\in L^1(\R_+;\dot B_{2,1}^{\frac{d}{2}+1}) , \\ (\varrho^\varepsilon-\overline{\varrho},v^\varepsilon)^{h, \varepsilon^{-1}}\in\mathcal{C}_b(\R_+;\dot B_{2,1}^{\frac{d}{2}-1})\cap L^1(\R_+;\dot B_{2,1}^{\frac{d}{2}+1}), \ w^\varepsilon\in \mathcal{C}_b(\R_+;\dot B_{2,1}^{\frac{d}{2}})\cap L^1(\R_+;\dot B_{2,1}^{\frac{d}{2}}) \bigg\}\cdotp
\end{multline*}
By studying the system satisfied by the difference between the solution of Euler-Poisson and that of Keller-Segel parabolic-elliptic and thanks to the previous theorem, we manage to justify that the solutions of the Euler-Poisson system that has been scaled back will converge to the solutions of the Keller-Segel system towards the following theorem :
\begin{theorem}\label{theoreme final}
We consider \eqref{système avant passage à la limite} for $\varepsilon>0$ small enough. Then, there exists a positive constant $\alpha$ (independent of $\varepsilon$) such that for all initial data $N_0\in \dot B_{2,1}^{\frac{d}{2}-1}\cap \dot B_{2,1}^\frac{d}{2}$ for \eqref{Keller-Segel} and $\tilde{Z}_0^\varepsilon\in \tilde{E}$ for \eqref{système avant passage à la limite} satisfying \begin{eqnarray}\label{condition initiale keller-segel}\|N_0\|_{\dot B_{2,1}^\frac{d}{2}\cap \dot B_{2,1}^{\frac{d}{2}-1}}\leq \alpha,\end{eqnarray}
$$\tilde{\mathcal{Z}}_0^{\varepsilon}\mathrel{\mathop:}= \|\tilde{\varrho}_0^\varepsilon-\overline{\varrho}\|_{\dot B_{2,1}^{\frac{d}{2}-1}}^{l}+ \|\tilde{\varrho}_0^\varepsilon-\overline{\varrho}\|_{\dot B_{2,1}^{\frac{d}{2}}}^{l^+,\ 1, \ \varepsilon^{-1}}+\varepsilon\|\tilde{v}_0^\varepsilon\|_{\dot B_{2,1}^{\frac{d}{2}}}^{l,\ \varepsilon^{-1}}+\varepsilon\|\big(\tilde{\varrho}_0^\varepsilon-\overline{\varrho},\varepsilon\tilde{v}_0^\varepsilon\big)\|_{\dot B_{2,1}^{\frac{d}{2}+1}}^{h,\ \varepsilon^{-1}}\leq \alpha, $$ the system \eqref{Keller-Segel} admits an unique solution $N$ in the space  $$\mathcal{C}_b\big(\R_+;\dot{B}_{2,1}^{\frac{d}{2}-2}\cap\dot B_{2,1}^\frac{d}{2}\big)\cap L^1\big(\R_+;\dot{B}_{2,1}^{\frac{d}{2}-2}\cap\dot B_{2,1}^\frac{d}{2}\big),$$ satisfying for all $t\geq 0$, \begin{eqnarray}\label{estimation keller-segel}\|N(t)-\overline{\varrho}\|_{\dot{B}_{2,1}^{\frac{d}{2}-1}\cap\dot B_{2,1}^\frac{d}{2}}+\int_0^t \|N-\overline{\varrho}\|_{\dot{B}_{2,1}^{\frac{d}{2}+2}\cap\dot B_{2,1}^{\frac{d}{2}+1}}d\tau \leq C \|N_0-\overline{\varrho}\|_{\dot{B}_{2,1}^{\frac{d}{2}-1}\cap\dot B_{2,1}^\frac{d}{2}},\end{eqnarray}
and the system \eqref{système avant passage à la limite} has an unique global-in-time solution $\tilde{Z}^\varepsilon$ in $\tilde{E}$ such that $$\tilde{\mathcal{Z}}(t)\leq C \tilde{\mathcal{Z}}_0$$ where \begin{multline}\label{espace fonctionnel final}
    \tilde{\mathcal{Z}}^\varepsilon(t)\mathrel{\mathop:}= \displaystyle\|\tilde{\varrho}^\varepsilon-\overline{\varrho}\|_{L^\infty\big(\dot B_{2,1}^{\frac{d}{2}-1}\big)}^{l}+\|\tilde{\varrho}^\varepsilon-\overline{\varrho}\|_{L^\infty\big(\dot B_{2,1}^{\frac{d}{2}}\big)}^{l^+,\ 1, \ \varepsilon^{-1} }+\varepsilon\|\tilde{v}^\varepsilon\|_{L^\infty \big(\dot B_{2,1}^{\frac{d}{2}}\big)}^{l, \  \varepsilon^{-1}} 
    +\varepsilon\|(\tilde{\varrho}^\varepsilon-\overline{\varrho},\varepsilon\tilde{v}^\varepsilon)\|_{L^\infty\big(\dot B_{2,1}^{\frac{d}{2}+1}\big)}^{h, \ \varepsilon^{-1}}
\\   +\|\tilde{\varrho}^\varepsilon-\overline{\varrho}\|_{L^1\big(\dot B_{2,1}^{\frac{d}{2}-1}\big)}^{l} + \|\tilde{v}^\varepsilon\|_{L^1\big(\dot B_{2,1}^{\frac{d}{2}}\big)}^{l}+\|\tilde{\varrho}^\varepsilon-\overline{\varrho}\|_{L^1\big(\dot B_{2,1}^{\frac{d}{2}+2}\big)}^{l^+,\ 1, \ \varepsilon^{-1}}+\|\tilde{v}^\varepsilon\|_{L^1\big(\dot B_{2,1}^{\frac{d}{2}+1}\big)}^{l^+,\ 1, \ \varepsilon^{-1}}
   \\
   + \varepsilon^{-1}\|(\tilde{\varrho}^\varepsilon-\overline{\varrho},\varepsilon\tilde{v}^\varepsilon)\|_{L^1\big(\dot B_{2,1}^{\frac{d}{2}+1}\big)}^{h, \varepsilon^{-1}} +\varepsilon\|\tilde{W}^\varepsilon\|_{L^\infty\big(\dot B_{2,1}^{\frac{d}{2}}\big)}+\varepsilon^{-1}\|\tilde{W}^\varepsilon\|_{L^1\big(\dot B_{2,1}^{\frac{d}{2}}\big)}
\end{multline} where $\displaystyle\tilde{W}^\varepsilon$ has been defined in \eqref{def mode amorti}.

~\\
Moreover, if, $\|N_0-\tilde{\varrho}^\varepsilon_0\|_{\dot B_{2,1}^{\frac{d}{2}-1}}\leq C\varepsilon$, then we have $$\|N-\tilde{\varrho}^\varepsilon\|_{L^\infty\big(\R_+;\dot B_{2,1}^{\frac{d}{2}-1}\big)}+\|N-\tilde{\varrho}^\varepsilon\|^h_{L^1\big(\R_+;\dot B_{2,1}^{\frac{d}{2}+1}\big)}+\|N-\tilde{\varrho}^\varepsilon\|^l_{L^1\big(\R_+;\dot B_{2,1}^{\frac{d}{2}}\big)}\leq C\varepsilon.$$
\end{theorem}

\begin{remark}
    This theorem ensures that the solution densities of the Euler-Poisson system converge (in the space highlighted in the theorem) towards the only solution of the Keller-Segel system. For the velocity limit, we can take the limit in \eqref{def mode amorti}.
\end{remark}

\section{Study of the Euler-Poisson system with damping} 
This section is devoted to the proof of theorem \ref{théorème Euler-Poisson}.

\subsection{Study of the linearized system}~~\\
Linearizing \eqref{système 2} around $(\overline{\varrho},0)$ yields the following system : 
\begin{eqnarray*}
\left\{ \begin{array}{ll} \partial_t \varrho+\dive v =0 \\ \partial_tv  +P'(\overline{\rho}) \nabla \varrho+v =- \varepsilon^2 \nabla \left(-\Delta\right)^{-1}\varrho. \end{array} \right. \end{eqnarray*} 

Performing the change of unknown $\tilde{\rho}\mathrel{\mathop:}=\varrho (t,P'(\overline{\rho})x)$ reduces the study to the case $P'(\overline{\rho})=1$ (after changing $\varepsilon$ into  $\displaystyle \varepsilon'\mathrel{\mathop:}= \frac{\varepsilon}{\sqrt{P'(\overline{\overline{\rho}})}}$). Hence we focus on the following linear system : 
\begin{eqnarray}\label{système linéarisé}
\left\{ \begin{array}{ll} \partial_t \varrho+\dive v =0, \\ \partial_tv  +\nabla \varrho+v =- \varepsilon^2 \nabla \left(-\Delta\right)^{-1}\varrho. \end{array} \right. \end{eqnarray}

By the Fourier transform, we get :
$$\frac{d}{dt}\begin{pmatrix} \widehat{\varrho} \\ \widehat{v}\end{pmatrix}+\begin{pmatrix} 0 & i\xi \\ i\left(1+\varepsilon^2\left|\xi\right|^{-2}\right)\xi^t & I_d \end{pmatrix}\begin{pmatrix} \widehat{\varrho}\\ \widehat{v}\end{pmatrix}=0.$$

The eigenvalues of the matrix of this system are:
\begin{itemize}
    \item[$\bullet$] $1$ is with multiplicity $d-1$,
    \item[$\bullet$] $ \displaystyle \left\{\begin{array}{ll} \displaystyle\lambda^{\pm}(\xi)=\frac{1}{2}(1\pm \sqrt{1-4(|\xi|^2+\varepsilon^2)}) &  \text{if} \ \displaystyle|\xi|^2+\varepsilon^2<\frac{1}{4} \\ & \\ \displaystyle\lambda^{\pm}(\xi)=\frac{1}{2}(1\pm i\sqrt{4(|\xi|^2+\varepsilon^2)-1}) &  \text{else} \end{array}\right.$ ~~~~~~~ are the two remaining eigenvalues.
\end{itemize}

For the high frequencies, we thus have : $$\Re\left(\lambda^{\pm}\left(\xi\right)\right)=1.$$ As for partially dissipative hyperbolic systems, we expect exponential decay for high frequencies.

For the low frequencies, in the case $\varepsilon<1/2,$ two regimes have to be considered: a very low frequency regime (i.e for $\displaystyle |\xi|\leq \sqrt{{1}/{4}-\varepsilon^2}$) and another for the medium frequencies (for $\displaystyle|\xi|\geq \sqrt{{1}/{4}-\varepsilon^2}$).

In what follows, we prove a priori estimates for  a smooth enough solution $(\varrho,v)$ of \eqref{système linéarisé}.   

\subsubsection{High frequency behavior: case $\displaystyle|\xi|\geq {1}/{2}$}

\begin{prop}
    Let $Z=(\varrho,v)$ and $\|Z\|_{\dot B_{2,1}^s}^h\mathrel{\mathop:}=\sum_{j\geq -1}2^{js}\|\dot \Delta_j Z\|_{L^2}$. Then we have : $$\boxed{\|Z(t)\|_{\dot B_{2,1}^{s}}^h+\int_0^t \|Z\|_{\dot B_{2,1}^{s}}^h d\tau \lesssim \|Z_0\|_{\dot B_{2,1}^{s}}^h}.$$
    \end{prop}

\begin{proof}

Since the classical energy method does not provide enough information, we consider, as in \cite{RD}, the evolution equation of $\nabla \varrho\cdot v$, namely $$\partial_t (\nabla \varrho\cdot v)=\nabla \partial_t \varrho\cdot v+\nabla \varrho\cdot \partial_t v=-\nabla \dive v\cdot v-\nabla \varrho\cdot \left(\nabla \varrho+v+\varepsilon^2 \nabla (-\Delta)^{-1} \varrho\right)\cdotp  $$

We then integrate on $\R^d$ and by integration by parts : $$\int_{\R^d} \partial_t(\nabla \varrho\cdot v) dx+\int_{\R^d} \nabla \varrho\cdot v dx=\|\dive v\|_{L^2}^2-\|\nabla \varrho\|_{L^2}^2 -\varepsilon^2 \|\varrho\|_{L^2}^2.$$

By taking the gradient in \eqref{système linéarisé} and taking the scalar product with $\nabla \varrho$ (respectively $\nabla v$), we get : 
$$\left\{ \begin{array}{l}
    \displaystyle\frac{1}{2}\frac{d}{dt}\|\nabla \varrho\|_{L^2}^2-\int_{\R^d}\nabla\nabla \varrho\cdot \nabla vdx=0   \\
     \displaystyle\frac{1}{2}\frac{d}{dt} \|\nabla v\|_{L^2}^2 + \int_{\R^d} \nabla \nabla \varrho \cdot \nabla vdx+\|\nabla v\|_{L^2}=-\varepsilon^2 \int_0^t \nabla \nabla (-\Delta)^{-1}\varrho\cdot \nabla v dx.
\end{array}\right.$$

These identities are also true for $(\varrho_j,v_j)$ (where $\varrho_j =\dot\Delta_j \varrho$ and $v_j=\dot \Delta_j v$) since the  studied system is linear with constant coefficients. 
To study high frequencies, we will further assume that  $j\in\N$. 

We then set $\displaystyle\mathcal{L}_j\mathrel{\mathop:}=\frac{1}{2}\|\nabla \varrho_j\|_{L^2}^2+\frac{1}{2}\|\nabla v_j\|_{L^2}^2+\frac{1}{4}\int_{\R^d} \nabla\varrho_j \cdot v_j \ dx$ which verifies : $$\frac{3}{8}\left(\|\nabla \varrho_j \|_{L^2}^2+\|\nabla v_j\|_{L^2}^2\right)\leq \mathcal{L}_j\leq \frac{5}{8}\left(\|\nabla \varrho_j \|_{L^2}^2+\|\nabla v_j\|_{L^2}^2\right)$$ because for $j\geq 0$, $\|v_j\|_{L^2}\leq \frac{1}{2}\|\nabla v_j\|_{L^2}$ by Bernstein's inequality.
\smallbreak
Then we have : $$\begin{aligned}\frac{d}{dt}\mathcal{L}_j+\|\nabla v_j\|_{L^2}^2+\frac{1}{4}\int_{\R^d} \nabla \varrho_j\cdot v_j dx-\frac{1}{4}\|\dive v\|_{L^2}^2+\frac{1}{4}\|\nabla \varrho_j\|_{L^2}^2&\leq \varepsilon^2 \|\varrho_j\|_{L^2}\|\nabla v_j\|_{L^2}+\frac{\varepsilon^2}{4}\|\varrho_j\|_{L^2}^2 \\ &\leq \frac{2\varepsilon^2}{4}\|\nabla\varrho_j\|_{L^2}^2+\frac{2\varepsilon^2}{4}\|\nabla v_j\|_{L^2}^2.
\end{aligned}$$

By the inequalities of Cauchy-Schwarz and Bernstein, we have : 
$$\begin{aligned}\|\nabla v_j\|_{L^2}^2+\frac{1}{4}\int_{\R^d} \nabla \varrho_j\cdot v_j dx-\frac{1}{4}\|\dive v\|_{L^2}^2&+\frac{1}{4}\|\nabla \varrho_j\|_{L^2}^2-\frac{2\varepsilon^2}{4}\|\nabla\varrho_j\|_{L^2}^2+\frac{2\varepsilon^2}{4}\|\nabla v_j\|_{L^2}^2 \\ &\geq -\frac{1}{4}\|\nabla \varrho_j\|_{L^2}\|v_j\|_{L^2}+\frac{3-2\varepsilon^2}{4}\|\nabla v_j\|_{L^2}^2+\frac{1-2\varepsilon^2}{4}\|\nabla \varrho_j\|_{L^2}^2\\ &\geq \frac{4-4\varepsilon^2}{8}\|\nabla v_j\|_{L^2}^2+\frac{1-4\varepsilon^2}{8}\|\nabla \varrho_j\|_{L^2}^2\\ &\geq\frac{1-4\varepsilon^2}{5}\mathcal{L}_j.\end{aligned}$$

By the inequality on $\mathcal{L}_j$, we thus  get : $$\frac{d}{dt}\mathcal{L}_j+\frac{1-4\varepsilon^2}{5}\mathcal{L}_j\leq 0.$$
For $\varepsilon\leq \frac{1}{4}$, so we get : $$\|\nabla (\varrho_j,v_j)(t)\|_{L^2}+\frac{3}{20}\int_0^t \|\nabla (\varrho_j,v_j)\|_{L^2}d\tau \leq  \|\nabla (\varrho_j,v_j)(0)\|_{L^2}.$$

We multiply by $2^{j(s-1)}$ and we sum up on $j\in\N$, we get then the inequality announced in the proposition (the case  $j=-1$
that presents no particular difficulty can be studied separately).
\end{proof}


\subsubsection{Low frequency behavior}
We have to proceed differently since the term $-\varepsilon^2 \nabla (-\Delta)^{-1}$ is of order $-1$. The goal will be here to understand his role.

~
\\
Let us use the Helmholtz decomposition to highlight the two behaviors corresponding respectively to the solenoid and the irrotational part of $v$: $$v=\mathcal{P}v+\mathcal{Q}v \qquad  \text{where } \ \mathcal{P} \ \text{ and } \ \mathcal{Q} \ \text{ verify } \ \mathcal{P}={\rm Id}+\nabla(-\Delta)^{-1}\dive \ \text{ and } \ \mathcal{Q}=-\nabla(-\Delta)^{-1}\dive. $$

~

First, let us look at the equation verified by $\mathcal{P}v$. By applying $\mathcal{P}$ to the second equation of the system  \eqref{système linéarisé} and using that $\mathcal{P}\nabla =0$, we get : $$\partial_t \mathcal{P}v+\mathcal{P}v=0.$$

By applying $\dot\Delta_j$ and taking the scalar product with $\mathcal{P}v_j$, we have: $$\frac{d}{dt}\|\mathcal{P}v_j\|_{L^2}^2+\|\mathcal{P}v_j\|_{L^2}^2=0.$$
Then we have by \eqref{lemme edo}: 

\begin{eqnarray}\label{estimation rot}
\boxed{\|\mathcal{P}v_j(t)\|_{L^2}+\int_0^t \|\mathcal{P}v_j\|_{L^2} \ d\tau\leq \|\mathcal{P}v_{0,j}\|_{L^2}.}
\end{eqnarray}
~ 
\\
Now let us look at the system satisfied by the divergence of $v$ and $\varrho$. We then set $u=\dive v$. By taking the divergence in the second equation of \eqref{système linéarisé}, we get the following $2\times 2$ system : 

$$\left\{ \begin{array}{l}  
     \partial_t \varrho+u=0, \\ \partial_t u + \Delta \varrho +u=\varepsilon^2 \varrho.
\end{array}\right.$$ ~\\
In Fourier variables, it becomes ~:

$$\left\{ \begin{array}{l}  \label{système linéaire pour amorti}
     \partial_t \widehat{\varrho}+\widehat{u}=0, \\ \partial_t \widehat{u}  -(|\xi|^2+\varepsilon^2) \widehat{\varrho} +\widehat u=0.\end{array}\right.$$

Setting $\displaystyle\widehat{w}\mathrel{\mathop:}= \frac{1}{\sqrt{|\xi|^2+\varepsilon^2}}\widehat{u}$ yields : 

\begin{eqnarray}\label{linéarisé tilde}\partial_t \tilde{Z} + A(D) \tilde{Z}+B(D) \tilde{Z} =0\end{eqnarray} $$\text{where} \quad \tilde{Z}=\begin{pmatrix} \varrho \\ w \end{pmatrix}, \quad A(\xi)=\begin{pmatrix} 0 & \sqrt{|\xi|^2+\varepsilon^2} \\ -\sqrt{|\xi|^2+\varepsilon^2} & 0  \end{pmatrix},\quad B(\xi)=\begin{pmatrix} 0 & 0 \\ 0 & 1 \end{pmatrix}\cdotp$$

Let us build by hand a Lyapunov functional, allowing us to recover the dissipative properties of the system on $\tilde{Z}$. By taking the scalar product with $\tilde{Z}$ in \eqref{linéarisé tilde} and looking at the time derivative of $\Re\left(\widehat{\varrho}\cdot \widehat{w}\right)$, we have:
$$\left\{ \begin{array}{l}
    \displaystyle \frac{1}{2}\frac{d}{dt}|\widehat{\varrho}|^2+\sqrt{|\xi|^2+\varepsilon^2}\ \Re(\widehat{\varrho}\cdot \widehat{w})=0, \\ \\ \displaystyle
      \frac{1}{2}\frac{d}{dt}|\widehat{w}|^2-\sqrt{|\xi|^2+\varepsilon^2} \ \Re(\widehat{\varrho}\cdot \widehat{w})+|\widehat{w}|^2=0,
      \\ \\\displaystyle \frac{d}{dt}\Re(\widehat{\varrho}\cdot \widehat{w})=-\sqrt{|\xi|^2+\varepsilon^2} \ |\widehat{w}|^2+\sqrt{|\xi|^2+\varepsilon^2} \ |\widehat{\varrho}|^2-\Re(\widehat{\varrho}\cdot \widehat{w}).
\end{array}\right.$$

With these equations, we can easily deduce the Lyapunov functional and the equation it verifies:

$$\frac{d}{dt}\left(|\widehat{\varrho}|^2+|\widehat{w}|^2-\sqrt{|\xi|^2+\varepsilon^2}\ \Re(\widehat{\varrho}\cdot \widehat{w})\right)+(2-(|\xi|^2+\varepsilon^2))|\widehat{w}|^2+(|\xi|^2+\varepsilon^2)|\widehat{\varrho}|^2-\sqrt{|\xi|^2+\varepsilon^2}\Re(\widehat{\varrho}\cdot \widehat{w})=0.$$
~

Yet $$ \sqrt{|\xi|^2+\varepsilon^2}\ \Re(\widehat{\varrho}\cdot \widehat{w})\leq \frac{|\xi|^2+\varepsilon^2}{2}|\widehat{\varrho}|^2+\frac{|\widehat{w}|^2}{2}\cdotp$$

Thus for $\displaystyle|\xi|^2+\varepsilon^2\leq 1$, we have: $$\frac{d}{dt}\left(|\widehat{\varrho}|^2+|\widehat{w}|^2-\sqrt{|\xi|^2+\varepsilon^2}\ \Re(\widehat{\varrho}\cdot \widehat{w})\right)+\frac{|\xi|^2+\varepsilon^2}{2}|\widehat{w}|^2+\frac{|\xi|^2+\varepsilon^2}{2}|\widehat{\varrho}|^2\leq 0.$$

We have then: \begin{eqnarray} \label{bf partie1}
\left|\left(\widehat{\varrho},\widehat{w}\right)(t,\xi)\right|\leq 2 e^{-\frac{1}{8}(|\xi|^2+\varepsilon^2) t}\left |\left(\widehat{\varrho}_0,\widehat{w}_0\right)(\xi)\right|.
\end{eqnarray}

So we have after spectral localization of the system by means of $\dot\Delta_j$
with $j\leq-1$ and for $\varepsilon$ small enough: \begin{eqnarray} \label{bf partie2} \hspace{1cm}
\|(\varrho_j,w_j)(t)\|_{L^2}+\frac{1}{8}(2^{2j}+\varepsilon^2)\int_0^t\| (\varrho_j,w_j)\|_{L^2} \,d\tau\leq 2 \|(\varrho_{0,j},w_{0,j}) \|_{L^2}.
\end{eqnarray}

~

We have a priori estimate on $\tilde{Z_j}$. A similar estimate has yet to be obtained for $(\varrho_j,u_j)$.

By definition of $w$ and by multiplying \eqref{bf partie2} by $\displaystyle\sqrt{2^{2j}+\varepsilon^2}$, we have as an estimate : 

\begin{eqnarray*}
 \boxed{\|(\sqrt{2^{2j}+\varepsilon^2}\ \varrho_j,u_j)(t)\|_{L^2}+\left(2^{2j}+\varepsilon^2\right)\int_0^t\| (\sqrt{2^{2j}+\varepsilon^2} \ \varrho_j,u_j)\|_{L^2} \,d\tau\lesssim  \|(\sqrt{2^{2j}+\varepsilon^2} \ \varrho_{0,j},u_{0,j}) \|_{L^2}.}
\end{eqnarray*}

~ 

These estimates reveal  two distinct regimes within the low frequencies : $|\xi|\leq \varepsilon$ ("very low frequencies") and $|\xi| \geq \varepsilon$ ("medium frequencies"). For very low frequencies, we have $2^{2j}+\varepsilon^2\simeq \varepsilon^2$, and thus $$\displaystyle \|(\varepsilon\tilde{\varrho}_j,u_j)(t)\|_{L^2}+\varepsilon^2\int_0^t \|(\varepsilon\tilde{\varrho}_j, u_j)\|_{L^2}d\tau \lesssim \|(\varepsilon\tilde{\varrho}_{0,j},u_{0,j})\|_{L^2}\quad\hbox{with}\quad  u=\dive v,$$
and for the medium frequencies, since $2^{2j}+\varepsilon^2\simeq 2^{2j}$,  
$$\displaystyle\|(\tilde{\varrho}_j,v_j)(t)\|_{L^2}+2^{2j}\int_0^t \|(\tilde{\varrho}_j,v_j)\|_{L^2}d\tau \lesssim  \|(\tilde{\varrho}_{0,j},v_{0,j})\|_{L^2}.$$

Consequently, if we denote : 
\begin{eqnarray}\label{déf normes}
\|Z\|_{\dot B_{2,1}^s}^{l^-,\varepsilon}\mathrel{\mathop:}=\sum\limits_{\underset{2^j\leq \varepsilon}{j\leq -1}}2^{js}\|Z_j\|_{L^2} \quad \text{et} \quad \|Z\|_{\dot B_{2,1}^s}^{l^+,\varepsilon}\mathrel{\mathop:}=\sum\limits_{\underset{2^j\geq \varepsilon}{j\leq -1}}2^{js}\|Z_j\|_{L^2}.    
\end{eqnarray}

then, we obtain:
$$\boxed{\left\{\begin{array}{l}
\displaystyle\|(\varepsilon\tilde{\varrho},\dive(v))(t)\|_{\dot B_{2,1}^s}^{l^-,\varepsilon}+\varepsilon^2 \int_0^t \|(\varepsilon\tilde{\varrho},\dive(v))\|_{\dot B_{2,1}^s}^{l^-,\varepsilon}d\tau \lesssim  \|\left(\varepsilon\tilde{\varrho}_0,\dive(v_0)\right)\|_{\dot B_{2,1}^s}^{l^-,\varepsilon}
\\
\displaystyle \|(\tilde{\varrho},v)(t)\|_{\dot B_{2,1}^s}^{l^+,\varepsilon} +\int_0^t  \|(\tilde{\varrho},v)\|_{\dot B_{2,1}^{s+2}}^{l^+,\varepsilon}d\tau \lesssim \|(\tilde{\varrho}_0,v_0)\|_{\dot B_{2,1}^s}^{l^+,\varepsilon},
\\ \displaystyle \|{\mathcal P}v(t)\|_{\dot B_{2,1}^{s}}^{l} +\int_0^t  \|{\mathcal P} v\|_{\dot B_{2,1}^{s}}^{l}d\tau \lesssim \|{\mathcal P}v_0\|_{\dot B_{2,1}^{s}}^{l} \qquad \text{(incompressible part)}.
\end{array} \right.}$$

\subsubsection{Damped mode and improvement of estimates for $v$}~~\\
Like in \cite{RD}, we consider the damped mode :
$$\tilde{W}\mathrel{\mathop:}=-\partial_t v =\nabla \varrho+\varepsilon^2 \nabla (-\Delta)^{-1}\varrho+v.$$

We have  : $$\partial_t \tilde{W}+\tilde{W}=-(\nabla+\varepsilon^2 \nabla(-\Delta)^{-1})\dive v.$$
By applying the localization operator $\dot \Delta_j$, taking the scalar product with $\tilde{W}_j$, multiplying by $2^{js}$, summing up on $j$ corresponding to very low and medium frequencies and applying the lemma \ref{lemme edo}, we get : $$\left\{\begin{array}{l}
     \displaystyle\|\tilde{W}(t)\|_{\dot B_{2,1}^s}^{l^-, \varepsilon}+\int_0^t \|\tilde{W}\|_{\dot B_{2,1}^s}^{l^-,\varepsilon}d\tau \lesssim \|\tilde{W}_0\|_{\dot B_{2,1}^{s}}^{l^-,\varepsilon}+\int_0^t \varepsilon^2 \|v\|_{\dot B_{2,1}^s}^{l^-,\varepsilon}d\tau  \\
      \displaystyle\|\tilde{W}(t)\|_{\dot B_{2,1}^s}^{l^+, \varepsilon}+\int_0^t \|\tilde{W}\|_{\dot B_{2,1}^s}^{l^+,\varepsilon}d\tau \lesssim \|\tilde{W}_0\|_{\dot B_{2,1}^{s}}^{l^+,\varepsilon}+\int_0^t  \|v\|_{\dot B_{2,1}^{s+2}}^{l^+,\varepsilon}d\tau
\end{array} \right. .$$
In particular, we have : $$\left\{\begin{array}{l}
     \displaystyle \|\tilde{W}_0\|_{\dot B_{2,1}^s}^{l^-,\varepsilon}\lesssim \varepsilon^2 \|\varrho_0\|_{\dot B_{2,1}^{s-1}}^{l^-,\varepsilon}+\|v_0\|_{\dot B_{2,1}^s}^{l^-,\varepsilon}  \\
      \|\tilde{W}_0\|_{\dot B_{2,1}^s}^{l^+,\varepsilon}\lesssim \|\varrho_0\|_{\dot B_{2,1}^{s+1}}^{l^+,\varepsilon}+\|v_0\|_{\dot B_{2,1}^s}^{l^+,\varepsilon}
\end{array}\right. .$$
By using the fact $v=\tilde{W}-\nabla \varrho-\varepsilon^2\nabla (-\Delta)^{-1}\varrho$ and the estimates on the low frequencies obtained previously, we have : $$\left\{\begin{array}{l} 
\displaystyle\|v(t)\|_{\dot B_{2,1}^s}^{l^-,\varepsilon}\leq  \|\tilde{W}(t)\|_{\dot B_{2,1}^s}^{l^-,\varepsilon}+\varepsilon^2\|\varrho(t)\|_{\dot B_{2,1}^{s-1}}^{l^-,\varepsilon}\lesssim \varepsilon^2 \|\varrho_0\|_{\dot B_{2,1}^{s-1}}^{l^-,\varepsilon}+\|v_0\|_{\dot B_{2,1}^s}^{l^-,\varepsilon}+\int_0^t \varepsilon^2 \|v\|_{\dot B_{2,1}^s}^{l^-,\varepsilon}d\tau
\\
\displaystyle
\int_0^t \varepsilon \|v\|_{\dot B_{2,1}^s}^{l^-,\varepsilon}d\tau \leq \int_0^t \varepsilon \|\tilde{W}\|_{\dot B_{2,1}^{s}}^{l^-,\varepsilon}d\tau+\int_0^t \varepsilon^3\|\varrho\|_{\dot B_{2,1}^{s-1}}^{l^-,\varepsilon}d\tau \lesssim \varepsilon \|\varrho_0\|_{\dot B_{2,1}^{s-1}}^{l^-,\varepsilon}+\|v_0\|_{\dot B_{2,1}^s}^{l^-,\varepsilon}+\int_0^t \varepsilon^2 \|v\|_{\dot B_{2,1}^s}^{l^-,\varepsilon}d\tau
\\ 
\displaystyle \|v(t)\|_{\dot B_{2,1}^s}^{l^+,\varepsilon}\lesssim \|\tilde{W}(t)\|_{\dot B_{2,1}^s}^{l^+,\varepsilon}+\|\varrho(t)\|_{\dot B_{2,1}^{s+1}}^{l^+,\varepsilon}\lesssim \|\varrho_0\|_{\dot B_{2,1}^{s+1}}^{l^+,\varepsilon}+\|v_0\|_{\dot B_{2,1}^s}^{l^+,\varepsilon}+\int_0^t  \|v\|_{\dot B_{2,1}^{s+2}}^{l^+,\varepsilon}d\tau
\\ 
\displaystyle \int_0^t \|v\|_{\dot B_{2,1}^{s+1}}^{l^+,\varepsilon}d\tau \leq \int_0^t \|\tilde{W}\|_{\dot B_{2,1}^{s+1}}^{l^+,\varepsilon}d\tau+\int_0^t \|\varrho\|_{\dot B_{2,1}^{s+2}}^{l^+,\varepsilon}d\tau \lesssim \|\varrho_0\|_{\dot B_{2,1}^{s}}^{l^+,\varepsilon}+\|v_0\|_{\dot B_{2,1}^s}^{l^+,\varepsilon}+\int_0^t  \|v\|_{\dot B_{2,1}^{s+2}}^{l^+,\varepsilon}d\tau
\end{array}
\right. $$
By summing up these inequalities and noticing that some terms in the right-hand side are negligible compared to those of the 
left-hand side, we get : $$\boxed{\left\{\begin{array}{l}
     \displaystyle \|v(t)\|_{\dot B_{2,1}^s}^{l^-,\varepsilon}+\int_0^t \varepsilon \|v\|_{\dot B_{2,1}^s}^{l^-,\varepsilon}d\tau \lesssim \varepsilon \|\varrho_0\|_{\dot B_{2,1}^{s-1}}^{l^-,\varepsilon}+\|v_0\|_{\dot B_{2,1}^s}^{l^-,\varepsilon}
    
      \\
    \displaystyle   \|v(t)\|_{\dot B_{2,1}^s}^{l^+,\varepsilon}+\int_0^t \|v\|_{\dot B_{2,1}^{s+1}}^{l^+,\varepsilon}d\tau \lesssim \|\varrho_0\|_{\dot B_{2,1}^{s}}^{l^+,\varepsilon}+\|v_0\|_{\dot B_{2,1}^s}^{l^+,\varepsilon}
\end{array}\right.}$$
\subsection{A priori estimates for the non-linear system}~\\
Let us now prove similar estimates for the non-linear system. To do this, we use Makino symmetrization,
which  consists in setting \begin{eqnarray} \label{changement de variable} c\mathrel{\mathop:}= \frac{\left(\gamma A\right)^{\frac{1}{2}}}{\tilde{\gamma}}\varrho^{\tilde{\gamma}} \ \text{with} \ \tilde{\gamma}=\frac{\gamma-1}{2}\cdotp \end{eqnarray}
After this change of unknown, we obtain:
\begin{eqnarray}\label{système avant théorème d'existence}\left\{ \begin{array}{ll} \partial_t c + v\cdot \nabla c +\tilde{\gamma} c \dive(v)=0 \\ \partial_t v +v\cdot \nabla v + \tilde{\gamma} c \nabla c + v=- \varepsilon^2\nabla \left(-\Delta\right)^{-1}\left( \left(\frac{\tilde{\gamma}}{\left(\gamma A \right)^{\frac{1}{2}}}\right)^{\frac{1}{\tilde{\gamma}}} c^{\frac{1}{\tilde{\gamma}}}-\overline{\varrho}\right). \end{array} \right. \end{eqnarray}

We set $\displaystyle f(x)\mathrel{\mathop:}=\left(\frac{\tilde{\gamma}}{(\gamma A)^{\frac{1}{2}}}\right)^{\frac{1}{\tilde{\gamma}}} x^\frac{1}{\tilde{\gamma}}.$
     So we have by Taylor’s formula with integral rest : $$\left(\frac{\tilde{\gamma}}{(\gamma A)^{\frac{1}{2}}}\right)^\frac{1}{\tilde{\gamma}}\left(c^\frac{1}{\tilde{\gamma}}-\overline{c}^\frac{1}{\tilde{\gamma}}\right)=f(\tilde{c}+\overline{c})-f(\overline{c})=\tilde{c} f'(\overline{c})+\int_{\overline{c}}^{c}{(c-y)}f''(y)dy.$$
     
     ~ 
     
    Let us set $\displaystyle F(\tilde{c})=\int_{\overline{c}}^{c}(c-y)f''(y)dy$ and $\displaystyle G(\tilde{c})=\tilde{c} f'(\overline{c})+\int_{\overline{c}}^{c}{(c-y)}f''(y)dy$ which vanishes at 0 where $\displaystyle \overline{c}\mathrel{\mathop:}=\frac{(\gamma A)^{\frac{1}{2}}}{\tilde{\gamma}}\ \overline{\varrho}^{\tilde{\gamma}}$ and $\tilde{c}=c-\overline{c}$.

~ 

Then we get: 
\begin{eqnarray}\label{système avec G}\left\{ \begin{array}{ll} \partial_t \tilde{c} + v\cdot \nabla \tilde{c }+\tilde{\gamma} \overline{c} \dive(v)+ \tilde{\gamma}\tilde{c}\dive(v)=0, \\ \partial_t v +v\cdot \nabla v + \tilde{\gamma} \overline{c} \nabla c + \tilde{\gamma} \tilde{c} \nabla c + v=- \varepsilon^2\nabla \left(-\Delta\right)^{-1}G(\tilde{c}). \end{array} \right. \end{eqnarray}

~ 
After changing $v(t,x)$ and $c(t,x)$
into $\displaystyle v(t,\tilde{\gamma}\overline{c}x)$  and $c(t,\tilde{\gamma}\overline{c}x)$, respectively, we can look at the following system (keeping the previous notations) :
\begin{eqnarray} \label{après changement de variable} \left\{ \begin{array}{ll}\displaystyle \partial_t \tilde{c} +\frac{1}{\tilde{\gamma}\overline{c}} v\cdot \nabla \tilde{c }+\dive(v)+ \frac{\tilde{c}}{\overline{c}}\dive(v)=0, \\\displaystyle \partial_t v +\frac{1}{\tilde{\gamma}\overline{c}} v\cdot \nabla v + \nabla c + \frac{\tilde{c}}{\overline{c}} \nabla c + v=- \varepsilon^2\nabla \left(-\Delta\right)^{-1}G(\tilde{c}). \end{array} \right. \end{eqnarray}
~ 

To simplify the presentation, suppose from now on that $\tilde{\gamma} \overline{c}=1$ : the general case works the same way.

We’re going to assume throughout this section that: \begin{eqnarray}\label{hypothèse de petitesse}\boxed{\|(\tilde{c},v) \|_{\dot B_{2,1}^\frac{d}{2}}\ll 1}.\end{eqnarray} 

In view of the linear analysis, we will start the following study with the choice of index: 
\begin{itemize}
    \item 
    $\displaystyle\frac{d}{2}-1$ for very low frequencies, 
    \item $\displaystyle\frac{d}{2}$ for medium frequencies, 
    \item $\displaystyle \frac{d}{2}+1$ for high frequencies.
    \end{itemize}    
    This choice of index is strongly inspired by the results of \cite{D3} where to study the relaxation limit, a similar choice is taken.

~
\\
Let us pose then : $$\tilde{\mathcal{L}}(t)\mathrel{\mathop:}=\|(\varepsilon \tilde{c},\dive v)(t)\|_{\dot B_{2,1}^{\frac{d}{2}-1}}^{l^-,\varepsilon}+\|(\tilde c, v)(t)\|_{\dot B_{2,1}^{\frac{d}{2}}}^{l^+,\varepsilon}+ \|\mathcal{P}v(t)\|_{\dot B_{2,1}^{\frac{d}{2}}}^{l} +\|(\tilde c, v)(t)\|_{\dot B_{2,1}^{\frac{d}{2}+1}}^{h}$$

and $$\tilde{\mathcal{H}}(t)\mathrel{\mathop:}=\varepsilon^2 \|(\varepsilon \tilde{c},\dive v)(t)\|_{\dot B_{2,1}^{\frac{d}{2}-1}}^{l^-,\varepsilon}+\|(\tilde c, v)(t)\|_{\dot B_{2,1}^{\frac{d}{2}+2}}^{l^+,\varepsilon}+\|\mathcal{P}v(t)\|_{\dot B_{2,1}^{\frac{d}{2}}}^{l}+\|(\tilde c, v)(t)\|_{\dot B_{2,1}^{\frac{d}{2}+1}}^{h}.$$

\begin{lemma}\label{Inégalités fonctionnelle}
    We have the following inequalities : $$\left\{\begin{array}{l}\|(\tilde{c},v)(t)\|_{\dot B_{2,1}^{\frac{d}{2}}}^{l}+\|(\tilde{c},v)(t)\|_{\dot B_{2,1}^{\frac{d}{2}+1}}^{h} \lesssim\tilde{\mathcal{L}}(t), \\ \|(\tilde{c},v)(t)\|_{\dot B_{2,1}^{\frac{d}{2}+2}}^{l}+\|(\tilde{c},v)(t)\|_{\dot B_{2,1}^{\frac{d}{2}+1}}^{h}\lesssim 
    \varepsilon^2 \|(\tilde{c},v)(t)\|_{\dot B_{2,1}^{\frac{d}{2}}}^{l^-,\varepsilon}+\|(\tilde{c},v)(t)\|_{\dot B_{2,1}^{\frac{d}{2}+2}}^{l^+,\varepsilon}+ \|(\tilde{c},v)(t)\|_{\dot B_{2,1}^{\frac{d}{2}+1}}^{h}\lesssim \tilde{\mathcal{H}}(t), \\
\varepsilon^2 \|(\tilde{c},v)(t)\|_{\dot B_{2,1}^{\frac{d}{2}}}^2+\|(\tilde{c},v)(t)\|_{\dot B_{2,1}^{\frac{d}{2}+1}}^2\lesssim \tilde{\mathcal{L}}(t)\tilde{\mathcal{H}}(t).\end{array}\right.$$
    
\end{lemma}
\begin{proof}~
The first two inequalities are easily deduced from the definition \eqref{déf normes}.
\\
Let us set $Z=(\tilde c,v)$. We have by definition and the second inequality:  $$\begin{aligned} \varepsilon^2 \|Z\|_{\dot B_{2,1}^{\frac{d}{2}}}^2 & =\left(\varepsilon^2\|Z\|_{\dot B_{2,1}^{\frac{d}{2}}}^{l^-,\varepsilon}+\varepsilon^2\|Z\|_{\dot B_{2,1}^{\frac{d}{2}}}^{l^+,\varepsilon}+\varepsilon^2\|Z\|_{\dot B_{2,1}^{\frac{d}{2}}}^{h}\right)\|Z\|_{\dot B_{2,1}^{\frac{d}{2}}} \\ & \lesssim \left(\varepsilon^2 \|Z\|_{\dot B_{2,1}^{\frac{d}{2}}}^{l^-,\varepsilon}+\|Z\|_{\dot B_{2,1}^{\frac{d}{2}+2}}^{l^+,\varepsilon}+\varepsilon^2 \|Z\|_{\dot B_{2,1}^{\frac{d}{2}+1}}^{h}\right)\|Z\|_{\dot B_{2,1}^{\frac{d}{2}}} \\ &  \lesssim \tilde{\mathcal{L}}\tilde{\mathcal{H}}.\hfill \end{aligned} $$

We have also : $$\begin{aligned}\|Z\|_{\dot B_{2,1}^{\frac{d}{2}+1}}^2 & \lesssim \left(\|Z\|_{\dot B_{2,1}^{\frac{d}{2}+1}}^{l^-,\varepsilon}\right)^2+\left(\|Z\|_{\dot B_{2,1}^{\frac{d}{2}+1}}^{l^+,\varepsilon}\right)^2+\left(\|Z\|_{\dot B_{2,1}^{\frac{d}{2}+1}}^{h}\right)^2 \\ & \lesssim \varepsilon^2 \|Z\|_{\dot B_{2,1}^{\frac{d}{2}}}^{l^-,\varepsilon}+ \|Z\|_{\dot B_{2,1}^{\frac{d}{2}}}^{l^+,\varepsilon}\|Z\|_{\dot B_{2,1}^{\frac{d}{2}+2}}^{l^+,\varepsilon}+\tilde{\mathcal{L}}\tilde{\mathcal{H}} \\  & \lesssim \tilde{\mathcal{L}}\tilde{\mathcal{H}}.\end{aligned}$$
\end{proof}
\subsubsection{High Frequency Estimates}~\\
We rely on the high-frequency analysis carried out in \cite{RD}.

\begin{prop}
    For high frequencies, we have the estimate:
    $$\boxed{\|Z(t)\|_{\dot B_{2,1}^{\frac{d}{2}+1}}^h+\int_0^t \|Z\|_{\dot B_{2,1}^{\frac{d}{2}+1}}^h d\tau \leq \|Z_0\|_{\dot B_{2,1}^{\frac{d}{2}+1}}^h+C\int_0^t \tilde{\mathcal{L}}(\tau)\tilde{\mathcal{H}}(\tau) d\tau. }$$
\end{prop}
\begin{proof}
Let us denote $\displaystyle \mathcal{L}_j\mathrel{\mathop:}=\frac{1}{2}\|Z_j\|_{L^2}^2+\frac{2^{-2j}}{4}\int_{\R^d}\nabla \tilde{c}_j\cdot v_j dx$.

The following lemma will enable us to handle the first term of $\mathcal{L}_j$:

\begin{lemma}\label{lemme hf partie 1}
We have the following inequality with $\displaystyle Z=\begin{pmatrix} \tilde{c} \\ v \end{pmatrix}$ and $\displaystyle\varepsilon'=\sqrt{f(\overline{c})} \ \varepsilon$  : 
$$\boxed{\frac{1}{2}\frac{d}{dt}\|Z_j\|_{L^2}^2+\|v_j\|_{L^2}^2+\int_0^t \varepsilon'^2 \nabla (-\Delta)^{-1} \tilde{c}_j\cdot v_j dx \leq C a_j 2^{-j\left(\frac{d}{2}+1\right)}\big(\|\nabla Z\|_{\dot B_{2,1}^{\frac{d}{2}}}\|Z\|_{\dot B_{2,1}^{\frac{d}{2}+1}}+\varepsilon^2 \|\tilde{c}\|_{\dot B_{2,1}^{\frac{d}{2}}}^2\big)\|Z_j\|_{L^2}}$$

with  $(a_j)$ verifying  \begin{eqnarray}\label{condition a_j} \sum_{j\in\Z}a_j=1.\end{eqnarray}
\end{lemma}

\begin{proof}
Let us apply $\dot\Delta_j$ to the system \eqref{après changement de variable}, then we get: 

$$\left\{ \begin{array}{l} \displaystyle
\partial_t c_j +v\cdot \nabla\tilde{c}_j+\frac{c}{\overline{c}}\dive(v_j)+\int_0^t \varepsilon'^2 \nabla (-\Delta)^{-1} \tilde{c}_j\cdot v_j dx=[v\cdot \nabla,\dot \Delta_j]\tilde{c}+[\frac{\tilde{c}}{\overline{c}},\dot \Delta_j]\dive(v),
\\ \displaystyle
\partial_t v_j +v\cdot \nabla v_j+\frac{c}{\overline{c}}\nabla c_j+v_j+\varepsilon'^2 \nabla (-\Delta)^{-1} \tilde{c}_j=-\varepsilon^2 \dot \Delta_j\nabla(-\Delta)^{-1}F(\tilde{c})+[v\cdot \nabla,\dot \Delta_j]v+[\frac{\tilde{c}}{\overline{c}},\dot \Delta_j]\nabla \tilde{c}.
\end{array} \right.$$
By performing the scalar product with $Z_j\mathrel{\mathop:}= \begin{pmatrix} \tilde{c}_j \\ v_j \end{pmatrix}$ in $L^2(\R^d;\R^n)$, we get: \begin{eqnarray*}\frac{1}{2}\frac{d}{dt}\|Z_j\|_{L^2}^2+\|v_j\|_{L^2}^2 & = & -\int_{\R^d} \left( v\cdot \nabla\tilde{c}_j+\frac{c}{\overline{c}}\dive(v_j)\right)\tilde{c_j}+\int_{\R^d}[v\cdot \nabla,\dot \Delta_j]\tilde{c}\tilde{c}_j dx 
\\ & & +\int_{\R^d}[\frac{\tilde{c}}{\overline{c}},\dot\Delta_j]\dive(v)\tilde{c}_j dx-\int_{\R^d}\left(v\cdot \nabla v_j+\frac{c}{\overline{c}}\nabla c_j\right)v_j dx \\ & & +\int_{\R^d} \left[v\cdot \nabla,\dot \Delta_j\right]v v_j dx +\int_{\R^d} [\frac{\tilde{c}}{\overline{c}},\dot\Delta_j]\nabla\tilde{c}\cdot v_j dx \\ & &
-\varepsilon^2 \int_{\R^d}\dot \Delta_j \nabla (-\Delta)^{-1} F(\tilde{c})\cdot v_j dx. \end{eqnarray*}

In order to bound the right-hand side, we use the following facts:

\begin{itemize}
    \item[$\bullet$]
$\displaystyle\|\varepsilon^2 \nabla (-\Delta)^{-1} F(\tilde{c})\|_{\dot B_{2,1}^{\frac{d}{2}+1}}\leq \varepsilon^2 \|F(\tilde{c})\|_{\dot B_{2,1}^{\frac{d}{2}}}
\leq \varepsilon^2 C(\|\tilde{c}\|_{L^\infty}) \|\tilde{c}\|_{\dot B_{2,1}^{\frac{d}{2}}}^2\lesssim \varepsilon^2 \|\tilde{c}\|_{\dot B_{2,1}^{\frac{d}{2}}}^2.$

\item[$\bullet$] The following commutator estimates with $\displaystyle s'=\frac{d}{2}+1$: \begin{eqnarray*}\|[v\cdot \nabla,\dot \Delta_j]Z\|_{L^2}\leq C a_j 2^{-j s'} \|\nabla Z\|_{\dot B_{2,1}^{\frac{d}{2}}} \|Z\|_{\dot B_{2,1}^{s'}} & \text{where} \ \sum_{j\in\Z} a_j=1, \\
\|[\frac{\tilde{c}}{\overline{c}},\dot \Delta_j]\dive(v)\|_{L^2}\leq C a_j 2^{-j s'} \|\nabla Z\|_{\dot B_{2,1}^{\frac{d}{2}}} \|Z\|_{\dot B_{2,1}^{s'}}, \\
\|[\frac{\tilde{c}}{\overline{c}},\dot \Delta_j]\nabla \tilde{c}\|_{L^2}\leq C a_j 2^{-j s'} \|\nabla Z\|_{\dot B_{2,1}^{\frac{d}{2}}} \|Z\|_{\dot B_{2,1}^{s'}}.\end{eqnarray*}

\item[$\bullet$] The integration by parts: 
$$\int_{\R^d} v\cdot \nabla \tilde{c}_j\tilde{c}_jdx=-\int_{\R^d}\frac{1}{2}\dive(v)|\tilde{c}_j|^2 dx.$$
\item[$\bullet$] $\displaystyle\int_{\R^d} v\cdot \nabla\tilde{v}_j\tilde{v}_j dx=-\int_{\R^d}\frac{1}{2}\dive(v)|\tilde{v}_j|^2 dx,$

\item[$\bullet$] $\displaystyle \int_{\R^d}\frac{c}{\overline{c}} \left(\dive(v_j)\tilde{c}_j+\nabla \tilde{c}_j v_j\right) dx=-\frac{1}{\overline{c}}\int_{\R^d} v_j \nabla(c \tilde{c}_j)dx+\frac{1}{\overline{c}}\int_{\R^d} c\nabla \tilde{c}_j v_jdx=-\int_{\R^d} v_j \tilde{c_j} \nabla \frac{c}{\overline{c}} dx,$

hence using the injection $\displaystyle \dot B_{2,1}^{\frac{d}{2}}(\R^d)\hookrightarrow \mathcal{C}_b(\R^d)$ \begin{eqnarray*}\int_{\R^d} v\cdot \nabla \tilde{c}_j\tilde{c}_jdx+\int_{\R^d} v\cdot \nabla\tilde{v}_j\tilde{v}_j dx+\int_{\R^d} \frac{c}{\overline{c}}\left(\dive(v_j)\tilde{c}_j+\nabla \tilde{c}_j v_j\right) dx \\ \leq C a_j 2^{-js'} \|\nabla Z\|_{\dot B_{2,1}^{\frac{d}{2}}} \|Z\|_{\dot B_{2,1}^{s'}} \|Z_j\|_{L^2}. \end{eqnarray*}
\end{itemize}
Whence the result by putting together all the above inequalities.
\end{proof}

For the other term of $\mathcal{L}_j$, look at \eqref{après changement de variable} as: 
$$\left\{\begin{array}{l} \displaystyle
     \partial_t \tilde{c} +\dive(v)=-v\cdot \nabla c -\frac{\tilde{c}}{\overline{c}}\dive(v)  
     \\ \displaystyle \partial_t v+\nabla \tilde{c}+\varepsilon'^2\nabla (-\Delta)^{-1} \tilde c+v=- v\cdot \nabla v -\frac{\tilde{c}}{\overline{c}} \nabla \tilde{c}-\varepsilon^2 \nabla (-\Delta)^{-1}F(\tilde{c})
\end{array} \right.$$  \begin{eqnarray}\label{epsilon'} \text{where} \  \varepsilon'=\sqrt{f(\overline{c})} \ \varepsilon. \end{eqnarray}

Let us denote $S_1\mathrel{\mathop:}=-v\cdot \nabla c -\frac{\tilde{c}}{\overline{c}}\dive(v) $ and $S_2\mathrel{\mathop:}=- v\cdot \nabla v -\frac{\tilde{c}}{\overline{c}} \nabla \tilde{c}-\varepsilon^2 \nabla (-\Delta)^{-1}F(\tilde{c})$ as well as $S\mathrel{\mathop:}=(S_1,S_2)$.

Analogously to linear analysis, we obtain: $$\displaylines{\frac{d}{dt}\left(\int_{\R^d}\nabla \tilde{c}_j\cdot v_j dx\right)+\int_{\R^d}\nabla \tilde{c}_j\cdot v_j dx-\|\dive v_j\|_{L^2}^2+\|\nabla \tilde{c}_j\|_{L^2}^2+\varepsilon'^2 \|\tilde{c}_j\|_{L^2}=\int_{\R^d}\Re(\dot \Delta_j S_1\cdot v)dx \hfill\cr\hfill+\int_{\R^d} \Re(\nabla c_j\cdot \dot\Delta_j S_2)dx.}$$

By Cauchy-Schwarz and Bernstein inequalities, we see  that $\displaystyle \mathcal{L}_j\mathrel{\mathop:}=\frac{1}{2}\|Z_j\|_{L^2}^2+\frac{2^{-2j}}{4}\int_{\R^d}\nabla \tilde{c}_j\cdot v_j dx$ verifies : 
\begin{equation}\label{Inequality on Lj}
\frac{3}{8}\|Z_j\|_{L^2}^2 \leq\mathcal{L}_j\leq \frac{1}{2}\|Z_j\|_{L^2}^2+\frac{2^{-2j}}{8}\|\nabla \tilde{c}_j\|_{L^2}^2+\frac{1}{8}\|v_j\|_{L^2}^2\leq \frac{5}{8}\|Z_j\|_{L^2}^2.\end{equation}
We then summarize the previous inequality with the inequality of the lemma \ref{lemme hf partie 1}: 
$$\displaylines{\frac{d}{dt}\mathcal{L}_j+\|v_j\|_{L^2}^2+\int_0^t \varepsilon'^2 \nabla (-\Delta)^{-1} \tilde{c}_j\cdot v_j dx-\frac{2^{-2j}}{4}\|\dive v_j\|_{L^2}^2+\frac{2^{-2j}}{4}\|\nabla c_j\|_{L^2}^2+\frac{\varepsilon'^2 2^{-2j}}{4}\|\tilde{c}_j\|_{L^2}^2+ \frac{2^{-2j}}{4} \int_{\R^d} \nabla c_j\cdot v_jdx \cr \leq C a_j 2^{-js'} \left(\|\nabla Z\|_{\dot B_{2,1}^{\frac{d}{2}}} \|Z\|_{\dot B_{2,1}^{s'}}+\varepsilon^2 \|\tilde{c}\|_{\dot B_{2,1}^{\frac{d}{2}}}^2\right) \|Z_j\|_{L^2} +\frac{2^{-2j}}{4}(\|\dot \Delta_j S_1\|_{L^2}\|v_j\|_{L^2}+\|\dot \Delta_j S_2\|_{L^2}\|\tilde{c}_j\|_{L^2}).}$$

But owing to the Cauchy-Schwarz and Bernstein inequalities, we have 
$$\displaystyle \int_0^t \varepsilon'^2 \nabla (-\Delta)^{-1} \tilde{c}_j\cdot v_j dx+\frac{2^{-2j}}{4}\int_{\R^d}\nabla \tilde{c}_j\cdot v_jdx\geq -\frac{1-2\varepsilon'^2}{8}\|\tilde{c}_j\|_{L^2}^2-\frac{1-2\varepsilon'^2}{8}\|v_j\|_{L^2}^2$$ which allows us to obtain, thanks to \eqref{Inequality on Lj},
$$\begin{aligned} \|v_j\|_{L^2}+\int_0^t \varepsilon'^2 \nabla (-\Delta)^{-1} \tilde{c}_j\cdot v_j dx-\frac{2^{-2j}}{4}\|\dive v_j\|_{L^2}^2+\frac{1}{4}\|c_j\|_{L^2}^2&+\frac{\varepsilon'^2 2^{-2j}}{4}\|\tilde{c}_j\|_{L^2}^2+\frac{2^{-2j}}{4} \int_{\R^d} \nabla c_j\cdot v_j \ dx \\ &\geq \frac{5-2\varepsilon'^2}{8}\|v_j\|_{L^2}^2+\frac{1-2\varepsilon'^2}{8}\|\tilde{c}_j\|_{L^2}^2\\&\geq \frac{1-2\varepsilon'^2}{5}\mathcal{L}_j.\end{aligned}$$ 
We then get for $\displaystyle\varepsilon'\leq {1}/{2}$ : 
\begin{eqnarray} \label{hf partie2}\frac{d}{dt}\mathcal{L}_j+\mathcal{L}_j\lesssim \left(2^{-j}\|\dot \Delta_j S\|_{L^2}+a_j 2^{-js'}(\|Z\|_{\dot B_{2,1}^{s'}}^2+\varepsilon^2 \|\tilde{c}\|_{\dot B_{2,1}^{\frac{d}{2}}}^2)\right)\|Z_j\|_{L^2}.
\end{eqnarray}

By decomposing $Z=Z^l+Z^h$ where $\widehat{Z^l}=\widehat Z \mathbb{1}_{|\xi|\leq 1}$ and $\widehat{Z^h}=\widehat Z \mathbb{1}_{|\xi|\geq 1}$ as well as using the lemmas \ref{Inégalités fonctionnelle} and \ref{Produit espace de Besov}, we have :

$$\displaystyle \left\{\begin{array}{l} \|Z \nabla Z\|_{\dot B_{2,1}^\frac{d}{2}}^h\leq \|Z \nabla Z^l\|_{\dot B_{2,1}^{\frac{d}{2}+1}}+\|Z \nabla Z^h\|_{\dot B_{2,1}^{\frac{d}{2}}}\lesssim \|Z\|_{\dot B_{2,1}^{\frac{d}{2}+1}}^2+\|Z\|_{\dot B_{2,1}^{\frac{d}{2}}}\|Z\|_{\dot B_{2,1}^{\frac{d}{2}+2}}^l+\|Z\|_{\dot B_{2,1}^{\frac{d}{2}}}\|Z\|_{\dot B_{2,1}^{\frac{d}{2}+1}}^h \lesssim \tilde{\mathcal{L}}\tilde{\mathcal{H}} \\ \displaystyle\|-\varepsilon^2 \nabla (-\Delta)^{-1} F(\tilde{c})\|_{\dot B_{2,1}^{\frac{d}{2}}}^h\lesssim \varepsilon^2 \|F(\tilde{c})\|_{\dot B_{2,1}^{\frac{d}{2}-1}}^h\leq \varepsilon^2 \|F(\tilde{c})\|_{\dot B_{2,1}^{\frac{d}{2}}}\lesssim \varepsilon^2 \|Z\|_{\dot B_{2,1}^{\frac{d}{2}}}^2\lesssim \tilde{\mathcal{L}}\tilde{\mathcal{H}}.\end{array}\right.$$

~

Consequently, we have : $$2^{js'}\sqrt{\mathcal{L}_j(t)}+c 2^{js'} \int_0^t \sqrt{\mathcal{L}_j} d\tau \leq 2^{js'} \sqrt{\mathcal{L}_j(0)}+C\int_0^t a_j \tilde{\mathcal{L}}(\tau)\tilde{\mathcal{H}}(\tau) d\tau, \quad  \text{with} \ s'=\frac{d}{2}+1$$

We deduce, by summing on $j\in\N$, the estimate for the high frequencies of the proposition.
\end{proof}

\subsubsection{Damped mode}
As explained in \cite{RD}, at the low frequency level, we have a loss of information on the part undergoing dissipation (here $v$). To overcome this loss, Danchin highlighted "the damped mode" to recover the missing information: we then pose $W\mathrel{\mathop:}=-\partial_t v$.

However here we cannot proceed as in linear analysis: some non-linear terms do not allow us to conclude.

So we are going to proceed differently here: we will first study the damped mode, which will allow us to have very precise estimates. In a second step, we will study the equation on $c$ using this damped mode and finally we will get the information on $v$ that we need using the estimates on $W$ and $c$.

From now on, let us take the constants (other than $\varepsilon$) appearing in our system equal to $1$ 
to simplify the  presentation: they play no role in what will follow.

Let us first look at this damped mode at the high frequencies.
\begin{lemma}\label{mode amorti hf}
We have the following estimate for the damped mode:
$$\|W(t)\|_{\dot B_{2,1}^{\frac{d}{2}}}^h+\int_0^t \|W\|_{\dot B_{2,1}^{\frac{d}{2}}}^h d\tau  \lesssim\displaystyle \tilde{\mathcal{L}}(t)+\int_0^t \tilde{\mathcal{H}}(\tau)d\tau+\int_0^t \tilde{\mathcal{L}}(\tau)\tilde{\mathcal{H}}(\tau)d\tau.$$
\end{lemma}
\begin{proof}
    By definition of $W$, we have : $$W= v+\nabla c + \varepsilon^2 \nabla(-\Delta)^{1}\tilde{c} +v\cdot v+\tilde{c}\nabla \tilde{c}+\varepsilon^2 \nabla (-\Delta)^{-1} F(\tilde{c}).$$

    We then have (by using the fact that $\displaystyle\dot B_{2,1}^{\frac{d}{2}}$ is a multiplicative algebra) : $$\displaylines{\|W(t)\|_{\dot B_{2,1}^{\frac{d}{2}}}^h\lesssim \|v(t)\|_{\dot B_{2,1}^{\frac{d}{2}}}^h+\|\tilde{c}(t)\|_{\dot B_{2,1}^{\frac{d}{2}+1}}^h+\varepsilon^2 \|\tilde{c}(t)\|_{\dot B_{2,1}^{\frac{d}{2}-1}}^h+\|v(t)\|_{\dot B_{2,1}^{\frac{d}{2}}}^h \|v(t)\|_{\dot B_{2,1}^{\frac{d}{2}+1}}^h +\|\tilde{c}(t)\|_{\dot B_{2,1}^{\frac{d}{2}}}^h \|\tilde{c}(t)\|_{\dot B_{2,1}^{\frac{d}{2}+1}}^h \hfill\cr\hfill +\varepsilon^2 \|\tilde{c}(t)\|_{\dot B_{2,1}^{\frac{d}{2}}}^h \|\tilde{c}(t)\|_{\dot B_{2,1}^{\frac{d}{2}-1}}^h.}$$

    By the lemma \ref{Inégalités fonctionnelle}, we have : $\displaystyle\|W(t)\|_{\dot B_{2,1}^{\frac{d}{2}}}^h\lesssim \tilde{\mathcal{L}}(t)$ and $\displaystyle \int_0^t \|W(\tau)\|_{\dot B_{2,1}^{\frac{d}{2}}}^h  d\tau\lesssim \int_0^t \tilde{\mathcal{H}}(\tau)d\tau+\int_0^t \tilde{\mathcal{L}}(\tau)\tilde{\mathcal{H}}(\tau)d\tau$.

\end{proof}

As $\partial_t v+W=0$, thus we have :
$$\begin{aligned}
 \partial_t W + W &  =  \partial_t v+W+\partial_t v \cdot \nabla v+v\cdot \nabla(\partial_t v)+\nabla (\partial_t c)+(\partial_t c) \nabla c +\tilde c \nabla(\partial_t c)+\varepsilon^2 \nabla(-\Delta)^{-1}\partial_t c \\ & \quad + \varepsilon^2 \nabla (-\Delta)^{-1}F'(\tilde c) \partial_t c  \\
& = -W\cdot \nabla v-v\cdot \nabla W-\nabla(v\cdot \nabla \tilde{c})-\nabla \dive v-\nabla(\tilde{c}\dive v)-v\cdot \nabla \tilde c \nabla \tilde c -\dive v \nabla \tilde c -\tilde c \dive v \nabla \tilde c 
\\ & \quad -\tilde c \nabla (v\cdot \nabla \tilde c)-\tilde c(\nabla \dive v)-\tilde c \nabla(\dive v \nabla \tilde c)-\varepsilon^2\nabla (-\Delta)^{-1}(v\cdot \nabla \tilde c)-\varepsilon^2 \nabla (-\Delta)^{-1}(\dive v)\\ & \quad -\varepsilon^2 \nabla (-\Delta)^{-1}(\tilde{c} \dive v).
\end{aligned}$$

\begin{lemma}\label{lemme mode amorti}
 The following estimate holds true:  \begin{equation}\label{mode amorti}
    \|W(t)\|_{\dot B_{2,1}^{\frac{d}{2}}}^l+\int_0^t \|W\|_{\dot B_{2,1}^{\frac{d}{2}}}^ld\tau  \lesssim\displaystyle \|W_0\|_{\dot B_{2,1}}^l+\int_0^t \tilde{\mathcal{L}}(\tau)\tilde{\mathcal{H}}(\tau)d\tau+\int_0^t \|v\|_{\dot B_{2,1}^{\frac{d}{2}+2}}^l d\tau+\varepsilon^2 \int_0^t \|\dive v\|_{\dot B_{2,1}^{\frac{d}{2}-1}}^ld\tau. \end{equation}
\end{lemma}
\begin{proof}
By applying $\dot \Delta_j$ to the previous equation verified by $W$, taking the scalar product with $W_j$, multiplying by $2^{j\frac{d}{2}}$, summing up on $j\in\Z^-$ and using the lemma \ref{lemme edo}, we get owing to the product laws of the lemma \ref{Produit espace de Besov} (each non-linear term in the right-hand side appears in the same order as the following estimate): 
    $$\displaylines{\|W(t)\|_{\dot B_{2,1}^{\frac{d}{2}}}^l +\int_0^t \|W\|_{\dot B_{2,1}^{\frac{d}{2}}}^l d\tau \lesssim \|W_0\|_{\dot B_{2,1}^{\frac{d}{2}}}^l+\int_0^t \|W\|_{\dot B_{2,1}^{\frac{d}{2}}} \|v\|_{\dot B_{2,1}^{\frac{d}{2}+1}} d\tau+\int_0^t \|v\cdot W\|_{\dot B_{2,1}^{\frac{d}{2}-1}}^l d\tau  
    \cr\hfill
    +\int_0^t \left(\|\nabla v\cdot \nabla \tilde c\|_{\dot B_{2,1}^{\frac{d}{2}}}+\|v\cdot \nabla\nabla c\|_{\dot B_{2,1}^{\frac{d}{2}}}^l\right)d\tau +\int_0^t \|v\|_{\dot B_{2,1}^{\frac{d}{2}+2}}^ld\tau  
    \cr\hfill
    +\int_0^t \left(\|\nabla \tilde c \dive v\|_{\dot B_{2,1}^{\frac{d}{2}}}+\|\tilde c \nabla \dive v\|_{\dot B_{2,1}^{\frac{d}{2}}} \right)d\tau+\int_0^t \|v\|_{\dot B_{2,1}^{\frac{d}{2}}}\|\tilde c\|_{\dot B_{2,1}^{\frac{d}{2}+1}}^2d\tau 
    \cr\hfill 
    +\int_0^t \|v\|_{\dot B_{2,1}^{\frac{d}{2}+1}}\|\tilde c\|_{\dot B_{2,1}^{\frac{d}{2}+1}}d\tau+\int_0^t \|\tilde c\|_{\dot B_{2,1}^{\frac{d}{2}+1}}\|v\|_{\dot B_{2,1}^{\frac{d}{2}+1}}\|\tilde c\|_{\dot B_{2,1}^{\frac{d}{2}}} d\tau 
    \cr\hfill
    +\int_0^t \|\tilde c \nabla(v\cdot \nabla \tilde c)\|_{\dot B_{2,1}^{\frac{d}{2}}}^ld\tau+ \int_0^t \|\tilde c(\nabla \dive v)\|_{\dot B_{2,1}^{\frac{d}{2}}}^ld\tau +\int_0^t \|\tilde c \nabla(\dive v \nabla \tilde c)\|_{\dot B_{2,1}^{\frac{d}{2}}}^ld\tau
    \cr\hfill
    +\varepsilon^2\int_0^t \|v\|_{\dot B_{2,1}^{\frac{d}{2}}}\|\tilde c\|_{\dot B_{2,1}^{\frac{d}{2}}}d\tau +\int_0^t \varepsilon^2 \|\dive v\|_{\dot B_{2,1}^{\frac{d}{2}-1}}^ld\tau +\varepsilon^2 \int_0^t \|\tilde c\|_{\dot B_{2,1}^{\frac{d}{2}}}\| v\|_{\dot B_{2,1}^{\frac{d}{2}}} d\tau.}$$

   Let us estimate one by one the terms appearing in the right-hand side: 
    \begin{itemize}
        \item[$\bullet$]  $\displaystyle \int_0^t \|W\|_{\dot B_{2,1}^{\frac{d}{2}}}^l \|v\|_{\dot B_{2,1}^{\frac{d}{2}+1}} d\tau$ can be absorbed by the left-hand side as we know that $\|Z\|_{L^\infty(\dot B_{2,1}^{\frac{d}{2}})}$ is small;
        \item[$\bullet$]  $\displaystyle \int_0^t \|W\|_{\dot B_{2,1}^{\frac{d}{2}}}^h \|v\|_{\dot B_{2,1}^{\frac{d}{2}+1}} d\tau\lesssim \int_0^t \tilde{\mathcal{L}}(\tau)\tilde{\mathcal{H}}(\tau)d\tau$ by the lemma \ref{mode amorti hf}; 
        \item[$\bullet$] $\displaystyle \int_0^t \|v\cdot W\|_{\dot B_{2,1}^{\frac{d}{2}-1}}^l d\tau\lesssim \int_0^t \|v\|_{\dot B_{2,1}^{\frac{d}{2}}}(\|W\|_{\dot B_{2,1}^{\frac{d}{2}}}^l+\|W\|_{\dot B_{2,1}^{\frac{d}{2}}}^h)d\tau$ has the low frequency part absorbed by the left-hand side and the other part below $\displaystyle \int_0^t \tilde{\mathcal{L}}(\tau)\tilde{\mathcal{H}}(\tau)d\tau$ by the lemma \ref{mode amorti hf};
        \item[$\bullet$] By using the lemma \ref{Inégalités fonctionnelle}
        and the fact that $\|Z\|_{L^\infty(\dot B_{2,1}^{\frac{d}{2}})}$ is small, we have : $$\displaylines{\int_0^t \bigg(\|\nabla v\cdot \nabla c\|_{\dot B_{2,1}^{\frac{d}{2}}}+\|\nabla \tilde c\cdot \dive v\|_{\dot B_{2,1}^{\frac{d}{2}}}+\|v\|_{\dot B_{2,1}^{\frac{d}{2}}}\|\tilde c\|_{\dot B_{2,1}^{\frac{d}{2}+1}}^2+ \|v\|_{\dot B_{2,1}^{\frac{d}{2}+1}}\|c\|_{\dot B_{2,1}^{\frac{d}{2}+1}} +\|\tilde c\|_{\dot B_{2,1}^{\frac{d}{2}+1}}\|v\|_{\dot B_{2,1}^{\frac{d}{2}+1}}\|\tilde c\|_{\dot B_{2,1}^{\frac{d}{2}}} \bigg)d\tau \hfill\cr\hfill \lesssim \int_0^t \|Z\|_{\dot B_{2,1}^{\frac{d}{2}+1}}^2 d\tau \lesssim \int_0^t \tilde{\mathcal{L}(\tau)}\tilde{\mathcal{H}}(\tau)d\tau;}$$

\item[$\bullet$] $\displaystyle \int_0^t \varepsilon^2 \|v\|_{\dot B_{2,1}^{\frac{d}{2}}}\|\tilde c\|_{\dot B_{2,1}^{\frac{d}{2}}}d\tau+\varepsilon^2 \int_0^t \|\tilde c\|_{\dot B_{2,1}^{\frac{d}{2}}}\|v\|_{\dot B_{2,1}^{\frac{d}{2}}}\lesssim \int_0^t \tilde{\mathcal{L}}(\tau)\tilde{\mathcal{H}}(\tau) d\tau $ by using the lemma \ref{Inégalités fonctionnelle};

\item[$\bullet$] $\|v\cdot \nabla \nabla c\|_{\dot B_{2,1}^{\frac{d}{2}}}^l\lesssim \|v\cdot \nabla \nabla c^l\|_{\dot B_{2,1}^{\frac{d}{2}}}+\|v\cdot \nabla \nabla c^h\|_{\dot B_{2,1}^{\frac{d}{2}}}\lesssim \|v\|_{\dot B_{2,1}^{\frac{d}{2}}}\|\tilde c\|_{\dot B_{2,1}^{\frac{d}{2}+2}}^l+\|v\cdot \nabla \nabla c^h\|_{\dot B_{2,1}^{\frac{d}{2}-1}}^l$ \\ 
$\lesssim \tilde{\mathcal{L}}\tilde{\mathcal{H}}+\|v\|_{\dot B_{2,1}^{\frac{d}{2}}}\|\tilde c\|_{\dot B_{2,1}^{\frac{d}{2}+1}}^h\lesssim \tilde{\mathcal{L}}\tilde{\mathcal{H}}$ by the lemmas \ref{Inégalités fonctionnelle} and \ref{Produit espace de Besov};

\item[$\bullet$] By proceeding in the same way as previously, we discover that : $$\|\tilde c\nabla \dive v\|_{\dot B_{2,1}^{\frac{d}{2}}}^l\lesssim \|\tilde c\|_{\dot B_{2,1}^{\frac{d}{2}}}\left(\|v\|_{\dot B_{2,1}^{\frac{d}{2}+2}}^l+\|v\|_{\dot B_{2,1}^{\frac{d}{2}+1}}^h\right)\lesssim \tilde{\mathcal{L}}\tilde{\mathcal{H}};$$

\item[$\bullet$] $\displaystyle \|\tilde c\nabla(v\cdot \nabla \tilde c)\|_{\dot B_{2,1}^{\frac{d}{2}}}^l\leq \|\tilde c \dive v \nabla \tilde c\|_{\dot B_{2,1}^{\frac{d}{2}}}+\|\tilde cv\cdot \nabla \nabla \tilde c\|_{\dot B_{2,1}^{\frac{d}{2}}}.$

For the first term, the lemma \ref{Inégalités fonctionnelle} is used and for the second, we do the same as for the previous two points. We find that this term is less than $\tilde{\mathcal{L}}\tilde{\mathcal{H}}$.

    \end{itemize}

    So we have inequality \eqref{mode amorti}.
\end{proof}

\subsubsection{Study of the low frequencies of $\tilde{c}$}~\\
We have : $$\partial_t \tilde c+\dive v=-\tilde c \dive v-v\cdot \nabla \tilde c.$$
However , $\displaystyle \dive v=\dive W-\Delta c+\varepsilon^2 \tilde c +\varepsilon^2 F(\tilde c)-\dive(\tilde c \nabla \tilde c)-\dive(v\cdot \nabla v)$.

So by rewriting the equation of $\tilde c$, we get $$\partial_t \tilde c-\Delta \tilde c+\varepsilon^2 c=-\dive W-\varepsilon^2 F(\tilde c) +\dive(\tilde c \nabla \tilde c)+\dive(v\cdot \nabla v)-\tilde c \dive v-v\cdot \nabla \tilde c$$

After spectral localization by means of  $\dot\Delta_j$,
 taking the scalar product with $\dot\Delta_j\tilde{c}$,  multiplying by
 $\varepsilon 2^{j(\frac{d}{2}-1)}$ (respectively $2^{j\frac{d}{2}}$) and, finally,  summing up on $2^j\leq \varepsilon$ (respectively $\varepsilon\leq 2^j \leq 1)$, we get 

\begin{lemma}\label{lemme estimation c} ~\\ The following estimates are satisfied by $\tilde{c}$ :
$$\displaylines{
        \|\varepsilon\tilde c(t)\|_{\dot B_{2,1}^{\frac{d}{2}-1}}^{l^-,\varepsilon}+\varepsilon^2 \int_0^t \|\varepsilon\tilde c(t)\|_{\dot B_{2,1}^{\frac{d}{2}-1}}^{l^-,\varepsilon}d\tau  \lesssim  \|\varepsilon\tilde c_0\|_{\dot B_{2,1}^{\frac{d}{2}-1}}^{l^-,\varepsilon}+\varepsilon\int_0^t \|W\|_{\dot B_{2,1}^{\frac{d}{2}}}^{l^-,\varepsilon}d\tau+\int_0^t \mathcal{C}(\tau)\tilde{\mathcal{C}}(\tau)d\tau
        \hfill\cr\hfill +\varepsilon\int_0^t \|v\cdot \nabla v \|_{\dot B_{2,1}^{\frac{d}{2}}}^{l^-,\varepsilon}+\varepsilon \int_0^t \|\tilde c \dive v\|_{\dot B_{2,1}^{\frac{d}{2}-1}}d\tau+\varepsilon\int_0^t \|v\cdot \nabla \tilde{c}\|_{\dot B_{2,1}^{\frac{d}{2}-1}}d\tau, }$$
$$\displaylines{\|\tilde{c}(t)\|_{\dot B_{2,1}^{\frac{d}{2}}}^{l^+,\varepsilon}+\int_0^t \|\tilde c\|_{\dot B_{2,1}^{\frac{d}{2}+2}}^{l^+,\varepsilon}d\tau  \lesssim \|\tilde c_0\|_{\dot B_{2,1}^{\frac{d}{2}}}^{l^+,\varepsilon}+\int_0^t \|W\|_{\dot B_{2,1}^{\frac{d}{2}+1}}^{l^+,\varepsilon}d\tau+ \int_0^t \mathcal{C}(\tau)\tilde{\mathcal{C}}(\tau)d\tau + \int_0^t \|\dive(v\cdot \nabla v)\|_{\dot B_{2,1}^{\frac{d}{2}}}^{l^+,\varepsilon}d\tau \hfill\cr\hfill  +\int_0^t \|\tilde c \dive v\|_{\dot B_{2,1}^{\frac{d}{2}}}d\tau+\int_0^t \|v\cdot \nabla\tilde c\|_{\dot B_{2,1}^{\frac{d}{2}}}d\tau}$$
where, in view of the linear estimates, we set $\mathcal{C}(t)\mathrel{\mathop:}= \|\varepsilon\tilde c(t)\|_{\dot B_{2,1}^{\frac{d}{2}-1}}^{l^-,\varepsilon}+\|\tilde c(t)\|_{\dot B_{2,1}^{\frac{d}{2}}}^{l^+,\varepsilon}+\|\tilde c(t)\|_{\dot B_{2,1}^{\frac{d}{2}+1}}^{h}$ and $\tilde{\mathcal{C}}(t)\mathrel{\mathop:}= \varepsilon^2  \|\varepsilon\tilde c(t)\|_{\dot B_{2,1}^{\frac{d}{2}-1}}^{l^-,\varepsilon}+\|\tilde c(t)\|_{\dot B_{2,1}^{\frac{d}{2}+2}}^{l^+,\varepsilon}+\|\tilde c(t)\|_{\dot B_{2,1}^{\frac{d}{2}+1}}^{h}$.
\end{lemma}
\subsubsection{Study of  the low frequencies of $v$}~~\\
It is now possible to deduce optimal bounds for $v$ from the ones we have just derived for $W$ and $\tilde{c}$. We need to decompose $v$ by using the damped mode as follows :
\begin{eqnarray}\label{v en fonction de w}
v=W-\nabla c+\varepsilon^2 \nabla(-\Delta)^{-1}\tilde c+v\cdot \nabla v+\tilde c \nabla c+\varepsilon^2 \nabla (-\Delta)^{-1}F(\tilde c)\end{eqnarray}

\begin{lemma}\label{lemme estimation v}
    Based on the following estimates, we will set:
    
    ~\\
    $\mathcal{V}(t)\mathrel{\mathop:}=\|v(t)\|_{\dot B_{2,1}^{\frac{d}{2}}}^l+\|v(t)\|_{\dot B_{2,1}^{\frac{d}{2}+1}}^h\;$ and $\;\tilde{\mathcal{V}}(t)\mathrel{\mathop:}=\varepsilon\|v(t)\|_{\dot B_{2,1}^{\frac{d}{2}}}^{l^-,\varepsilon}+\|v(t)\|_{\dot B_{2,1}^{\frac{d}{2}+1}}^{l^+,\varepsilon}+\|v(t)\|_{\dot B_{2,1}^{\frac{d}{2}+1}}^h$.

~\\
    We then obtain the following estimates: 
~

$\begin{aligned}
        \|v(t)\|_{\dot B_{2,1}^{\frac{d}{2}}}^{l^-,\varepsilon}+\varepsilon\int_0^t \|v\|_{\dot B_{2,1}^{\frac{d}{2}}}^{l^-,\varepsilon}d\tau & \lesssim  \|W(t)\|_{\dot B_{2,1}^{\frac{d}{2}}}^{l^-,\varepsilon}+(\varepsilon^2+\|Z(t)\|_{\dot B_{2,1}^{\frac{d}{2}+1}}+\varepsilon \mathcal{C}(t) ) \mathcal{C}(t)+ \mathcal{V}(t)\|Z(t)\|_{\dot B_{2,1}^{\frac{d}{2}+1}}\\ & +\int_0^t \varepsilon\|W\|_{\dot B_{2,1}^{\frac{d}{2}}}^{l^-,\varepsilon}d\tau +\int_0^t \tilde{\mathcal{C}}(\tau)d\tau+\int_0^t \mathcal{C}(\tau)\tilde{\mathcal{C}}(\tau)d\tau+\varepsilon\int_0^t \mathcal{V}(\tau)\tilde{\mathcal{V}}(\tau)d\tau.
    \end{aligned}$

$\begin{aligned}
    \|v(t)\|_{\dot B_{2,1}^{\frac{d}{2}}}^{l^+,\varepsilon}+\int_0^t \|v(t)\|_{\dot B_{2,1}^{\frac{d}{2}+1}}^{l^+,\varepsilon}d\tau  \lesssim & \|W(t)\|_{\dot B_{2,1}^{\frac{d}{2}}}^{l^+,\varepsilon}+\|\tilde c(t)\|_{\dot B_{2,1}^{\frac{d}{2}+1}}^{l^+,\varepsilon}+\left(\mathcal{V}(t)+\mathcal{C}(t)\right)\|Z(t)\|_{\dot B_{2,1}^{\frac{d}{2}+1}} +\varepsilon \left(\mathcal{C}(t)\right)^2 \\ & \quad +\int_0^t \|W\|_{\dot B_{2,1}^{\frac{d}{2}+1}}^{l^+,\varepsilon}d\tau +\int_0^t \tilde{\mathcal{C}}(\tau)d\tau+\int_0^t \mathcal{V}(\tau)\tilde{\mathcal{V}}(\tau)d\tau+\int_0^t \mathcal{C}(\tau)\tilde{\mathcal{C}}(\tau)d\tau.
\end{aligned}$
\end{lemma}
\begin{proof}

    By lemmas \ref{Produit espace de Besov}, \ref{fonction régulière besov} and \eqref{v en fonction de w}, we obtain : 
    
    $$ \|v(t)\|_{\dot B_{2,1}^{\frac{d}{2}}}^{l^-,\varepsilon}\lesssim \|W\|_{\dot B_{2,1}^{\frac{d}{2}}}^{l^-,\varepsilon}+\varepsilon^2 \|\tilde c\|_{\dot B_{2,1}^{\frac{d}{2}-1}}^{l^-,\varepsilon}+\|v\|_{\dot B_{2,1}^{\frac{d}{2}}}\|v\|_{\dot B_{2,1}^{\frac{d}{2}+1}}+\|\tilde c\|_{\dot B_{2,1}^{\frac{d}{2}}}\|\tilde c\|_{\dot B_{2,1}^{\frac{d}{2}+1}}+\varepsilon^2 \|\tilde c\|_{\dot B_{2,1}^{\frac{d}{2}}}\|\tilde c\|_{\dot B_{2,1}^{\frac{d}{2}-1}}.$$
    
   We deduce : $$\left\{\begin{array}{l}
   \|v(t)\|_{\dot B_{2,1}^{\frac{d}{2}}}^{l^-,\varepsilon} \lesssim \|W(t)\|_{\dot B_{2,1}^{\frac{d}{2}}}^{l^-,\varepsilon}+(\varepsilon+\|Z(t)\|_{\dot B_{2,1}^{\frac{d}{2}+1}}+\varepsilon \mathcal{C}(t) )\mathcal{C}(t)+\mathcal{V}(t)\|v(t)\|_{\dot B_{2,1}^{\frac{d}{2}+1}} \\
    \displaystyle\varepsilon\int_0^t \|v\|_{\dot B_{2,1}^{\frac{d}{2}}}^{l^-,\varepsilon}d\tau\lesssim \int_0^t \varepsilon\|W\|_{\dot B_{2,1}^{\frac{d}{2}}}^{l^-,\varepsilon}d\tau +\int_0^t \tilde{\mathcal{C}}(\tau)d\tau+\int_0^t \mathcal{C}(\tau)\tilde{\mathcal{C}}(\tau)d\tau+\varepsilon\int_0^t \mathcal{V}(\tau)\tilde{\mathcal{V}}(\tau)d\tau.
    \end{array}\right.$$

    Similarly, we get :  $$\|v(t)\|_{\dot B_{2,1}^{\frac{d}{2}}}^{l^+,\varepsilon}\lesssim \|W(t)\|_{\dot B_{2,1}^{\frac{d}{2}}}^{l^+,\varepsilon}+\|\tilde c(t)\|_{\dot B_{2,1}^{\frac{d}{2}+1}}^{l^+,\varepsilon}+\left(\mathcal{V}(t)+\mathcal{C}(t)\right)\|Z(t)\|_{\dot B_{2,1}^{\frac{d}{2}+1}}+\varepsilon \left(\mathcal{C}(t)\right)^2.$$

    By lemmas \ref{Produit espace de Besov} and \ref{fonction régulière besov} and inequality $\displaystyle \|Z\|_{\dot B_{2,1}^{\frac{d}{2}+1}}^{l^+,\varepsilon}\leq \|Z\|_{\dot B_{2,1}^{\frac{d}{2}}}^{l^+,\varepsilon} $, we have also : 
    $$\|v\|_{\dot B_{2,1}^{\frac{d}{2}+1}}^{l^+,\varepsilon}\lesssim \|W\|_{\dot B_{2,1}^{\frac{d}{2}+1}}^{l^+,\varepsilon}+\|\tilde c\|_{\dot B_{2,1}^{\frac{d}{2}+2}}^{l^+,\varepsilon}+\|v\|_{\dot B_{2,1}^{\frac{d}{2}}}\|v\|_{\dot B_{2,1}^{\frac{d}{2}+1}}+\|\tilde c \cdot \nabla c\|_{\dot B_{2,1}^{\frac{d}{2}+1}}^{l^+,\varepsilon}+\varepsilon^2 \|\tilde c\|_{\dot B_{2,1}^{\frac{d}{2}}}^2. $$

    Like the proof of lemma \ref{Inégalités fonctionnelle}, we have that $\displaystyle \varepsilon^2 \|\tilde c\|_{\dot B_{2,1}^{\frac{d}{2}}}^2\lesssim \mathcal{C}\tilde{\mathcal{C}} $.

   Moreover, $$\begin{aligned} \|\tilde c \cdot \nabla c\|_{\dot B_{2,1}^{\frac{d}{2}+1}}^{l^+,\varepsilon}&\lesssim \|\tilde c\cdot \nabla c^l\|_{\dot B_{2,1}^{\frac{d}{2}+1}}^l+\|\tilde c\cdot \nabla c^h\|_{\dot B_{2,1}^{\frac{d}{2}+1}}^l\\
   &\lesssim \|\tilde c\|_{\dot B_{2,1}^{\frac{d}{2}+1}}\|\tilde c\|_{\dot B_{2,1}^{\frac{d}{2}+1}}^l+\|\tilde c\|_{\dot B_{2,1}^{\frac{d}{2}}}\|\tilde c\|_{\dot B_{2,1}^{\frac{d}{2}+2}}^l+\|\tilde c \cdot \nabla c^h\|_{\dot B_{2,1}^{\frac{d}{2}}}^l\\& \lesssim \mathcal{C}\tilde{\mathcal{C}}+\|\tilde c\|_{\dot B_{2,1}^{\frac{d}{2}}}\|\tilde c\|_{\dot B_{2,1}^{\frac{d}{2}+1}}\lesssim \mathcal{C}\tilde{\mathcal{C}}.\end{aligned}$$

We deduce : $$\int_0^t \|v\|_{\dot B_{2,1}^{\frac{d}{2}}}^{l^+,\varepsilon}d\tau \lesssim \int_0^t \|W\|_{\dot B_{2,1}^{\frac{d}{2}+1}}^{l^+,\varepsilon}d\tau+\int_0^t \tilde{\mathcal{C}}(\tau)d\tau+\int_0^t \mathcal{V}(\tau)\tilde{\mathcal{V}}(\tau)d\tau+\int_0^t \mathcal{C}(\tau)\tilde{\mathcal{C}}(\tau)d\tau.$$
\end{proof}
\subsubsection{Final a priori estimates}~\\
Let us denote \begin{equation}\label{LH}
    \mathcal{L}(t)\mathrel{\mathop:}=\mathcal{C}(t)+\mathcal{V}(t)+\|W(t)\|_{\dot B_{2,1}^{\frac{d}{2}}}^l
    \ \hbox{ and }\ \mathcal{H}(t)\mathrel{\mathop:}=\tilde{\mathcal{C}}(t)+\tilde{\mathcal{V}}(t)+\|W(t)\|_{\dot B_{2,1}^{\frac{d}{2}}}^l.\end{equation}

We notice that : $\tilde{\mathcal{L}}(t)\leq \mathcal{L}(t)$ and $\tilde{\mathcal{H}}(t)\leq \mathcal{H}(t)$.

\begin{prop}\label{estimées finales} We have the following estimate :
$$\displaystyle \mathcal{L}(t)+\int_0^t \mathcal{H}(\tau)d\tau \lesssim \mathcal{L}(0)+\int_0^t \mathcal{L}(\tau)\mathcal{H}(\tau)d\tau.$$

If we take $\mathcal{L}(0)$ sufficiently small, thus we obtain the final a priori estimate : $$\mathcal{L}(t)+\int_0^t \mathcal{H}(\tau)d\tau \lesssim \mathcal{L}(0).$$

\end{prop}

\begin{proof}
First, we note that by summing up the previous inequalities (lemmas \ref{lemme mode amorti}, \ref{lemme estimation c}, \ref{lemme estimation v}), terms in the right-hand side can be absorbed by those of the left-hand side.
Indeed :
\begin{itemize}
    \item[$\bullet$] In \eqref{mode amorti}, we have that the term $\displaystyle\int_0^t \|v\|_{\dot B_{2,1}^{\frac{d}{2}+2}}^l d\tau+\varepsilon^2 \int_0^t \|\dive v\|_{\dot B_{2,1}^{\frac{d}{2}-1}}d\tau$ is negligible compared to $\displaystyle\int_0^t \tilde{\mathcal{V}}(\tau)d\tau$ (so also to $\displaystyle \int_0^t \mathcal{H}(\tau)d\tau$).
    \item[$\bullet$] In the estimates of the lemma \ref{lemme estimation c}, we have that $\displaystyle \int_0^t \varepsilon\|W\|_{\dot B_{2,1}^{\frac{d}{2}}}^{l^-,\varepsilon}d\tau$, $\displaystyle \int_0^t \|W\|_{\dot B_{2,1}^{\frac{d}{2}+1}}^{l^+,\varepsilon}d\tau$ are negligible compared to $\displaystyle\int_0^t \|W\|_{\dot B_{2,1}^{\frac{d}{2}}}^l d\tau$. By using that$\|Z\|_{\dot B_{2,1}^{\frac{d}{2}}}$ is small and the lemma \ref{Produit espace de Besov}, we also have that terms $\displaystyle \varepsilon \int_0^t \|v\cdot \nabla v\|_{\dot B_{2,1}^{\frac{d}{2}}}^{l^-,\varepsilon}d\tau$ and $\displaystyle \int_0^t \|\dive(c\cdot \nabla v)\|_{\dot B_{2,1}^{\frac{d}{2}}}^{l^+,\varepsilon}d\tau$ are negligible compared to $\displaystyle \int_0^t \tilde{\mathcal{V}}(\tau)d\tau$.
    \item[$\bullet$] $\displaystyle (\varepsilon^2+\varepsilon C(t))C(t)$ and $\displaystyle \|\tilde c(t)\|_{\dot B_{2,1}^{\frac{d}{2}+1}}^{l^+,\varepsilon}$ are negligible compared to $\mathcal{L}(t)$.
\end{itemize}
Using the definitions of the various introduced norms and $\mathcal{L}$, $\mathcal{H}$ and the lemma \ref{Produit espace de Besov}, we have :
\begin{itemize}
    \item[$\bullet$] $\displaystyle \int_0^t \tilde{\mathcal{L}}\tilde{\mathcal{H}}d\tau \lesssim \int_0^t \mathcal{L}\mathcal{H}d\tau$,
    \item[$\bullet$] $\displaystyle \int_0^t \mathcal{C}(\tau)\tilde{\mathcal{C}}(\tau)d\tau+\int_0^t \mathcal{V}(\tau)\tilde{\mathcal{V}}(\tau)d\tau\lesssim \int_0^t \mathcal{L}(\tau)\mathcal{H}(\tau)d\tau,$
    \item[$\bullet$] $\displaystyle \int_0^t \varepsilon \|\tilde c \dive v\|_{\dot B_{2,1}^{\frac{d}{2}-1}}^{l^-,\varepsilon}d\tau\lesssim \int_0^t \|\varepsilon\tilde c\|_{\dot B_{2,1}^{\frac{d}{2}-1}}\|v\|_{\dot B_{2,1}^{\frac{d}{2}+1}}d\tau \lesssim \int_0^t \mathcal{L}(\tau)\mathcal{H}(\tau)d\tau$,
    \item[$\bullet$] $\displaystyle \int_0^t \|\tilde c \dive v\|_{\dot B_{2,1}^{\frac{d}{2}}}^{l^-,\varepsilon}d\tau\lesssim \int_0^t \|\tilde c\|_{\dot B_{2,1}^{\frac{d}{2}}}\|v\|_{\dot B_{2,1}^{\frac{d}{2}+1}}d\tau \lesssim \int_0^t \mathcal{L}(\tau)\mathcal{H}(\tau)d\tau,$
    \item[$\bullet$] $\displaystyle \varepsilon\|v\cdot \nabla c\|_{\dot B_{2,1}^{\frac{d}{2}-1}}^{l^-,\varepsilon}d\tau \lesssim \int_0^t \|\tilde c\|_{\dot B_{2,1}^{\frac{d}{2}}}\|\varepsilon v\|_{\dot B_{2,1}^{\frac{d}{2}}}d\tau \lesssim \int_0^t \mathcal{L}(\tau)\mathcal{H}(\tau)d\tau.$
\end{itemize}

Now, if we sum up the previous inequalities by using what we just did before and by removing "negligible terms compared to the right term", we get: 
$$\mathcal{L}(t)+\int_0^t \mathcal{H}(\tau)d\tau\lesssim \mathcal{L}(0)+\int_0^t \mathcal{L}(\tau)\mathcal{H}(\tau)d\tau+\int_0^t \|v\cdot \nabla c\|_{\dot B_{2,1}^{\frac{d}{2}}}^{l^+,\varepsilon}d\tau.$$

To handle the last term, let us use the fact that $v$ and $w$ are interrelated as follows:
$$v=W-\nabla c+\varepsilon^2 \nabla(-\Delta)^{-1}\tilde c+v\cdot \nabla v+\tilde c \nabla c+\varepsilon^2 \nabla (-\Delta)^{-1}F(\tilde c).$$

By the lemma \ref{Produit espace de Besov}, we have also :
 $$\displaylines{ \|v\cdot \nabla c\|_{\dot B_{2,1}^{\frac{d}{2}}}\lesssim \|W\|_{\dot B_{2,1}^{\frac{d}{2}}}\|\tilde c\|_{\dot B_{2,1}^{\frac{d}{2}+1}}+\|\varepsilon \tilde c\|_{\dot B_{2,1}^{\frac{d}{2}-1}}\|\tilde c\|_{\dot B_{2,1}^{\frac{d}{2}+1}}+\|\tilde c\|_{\dot B_{2,1}^{\frac{d}{2}+1}}^2+\|v\|_{\dot B_{2,1}^{\frac{d}{2}}}\|v\|_{\dot B_{2,1}^{\frac{d}{2}+1}}\|\tilde c\|_{\dot B_{2,1}^{\frac{d}{2}+1}}   \hfill\cr\hfill +\|\tilde c\|_{\dot B_{2,1}^{\frac{d}{2}}}\|\tilde c\|_{\dot B_{2,1}^{\frac{d}{2}+1}}^2+\|\tilde c\|_{\dot B_{2,1}^{\frac{d}{2}}}\|\varepsilon \tilde c\|_{\dot B_{2,1}^{\frac{d}{2}-1}}\|\varepsilon \tilde c\|_{\dot B_{2,1}^{\frac{d}{2}+1}}.} $$
We then have by lemma \ref{Inégalités fonctionnelle} and definition of $\mathcal{L}$, $\mathcal{H}$ : 
$$ \|v\cdot \nabla c\|_{\dot B_{2,1}^{\frac{d}{2}}}\lesssim \tilde{\mathcal{L}}\tilde{\mathcal{H}} +\mathcal{L}\mathcal{H}\lesssim \mathcal{L}\mathcal{H},$$
hence the result.
We have the second inequality of the proposition by using lemma \ref{lemme edo2}.
\end{proof}

\subsection{A global well-posedness theorem}
Here is the theorem that we will prove in the rest of this section: 

\begin{theorem}\label{théorème d'existence sur le système après changement de variable}~\\
We assume $\displaystyle\varepsilon' \leq {1}/{2}$ with $\varepsilon'$ defined in \eqref{epsilon'}. Then, there exists a positive constant $\alpha$ such that for all 
$Z_0^\varepsilon=(\tilde{c}_0,v_0)\in \dot B_{2,1}^{\frac{d}{2}}\cap \dot B_{2,1}^{\frac{d}{2}+1} $ satisfying $$ \mathcal{Z}_0^{\varepsilon}\mathrel{\mathop:}=\displaystyle\|\varepsilon \tilde{c}_0\|_{\dot B_{2,1}^{\frac{d}{2}-1}}^{l^-,\varepsilon'}+\|\tilde{c_0}\|_{\dot B_{2,1}^{\frac{d}{2}}}^{l^+,\varepsilon'}+\|v_0\|_{\dot B_{2,1}^{\frac{d}{2}}}^{l}+\|(\tilde{c}_0,v_0)\|_{\dot B_{2,1}^{\frac{d}{2}+1}}^h\leq \alpha,$$ the system \eqref{système avant théorème d'existence} with the initial data $(c_0,v_0)$ admits a unique global-in-time solution $Z=(\tilde c,v)$ in the set \begin{multline*}
E\mathrel{\mathop:}=\bigg\{(\tilde{c},v) \ \bigg| \ \varepsilon\tilde{c}^{\ l^-,\varepsilon'}\in \mathcal{C}_b(\R_+:\dot B_{2,1}^{\frac{d}{2}-1}), \ \varepsilon^3\tilde{c}^{\ l^-,\varepsilon'}\in L^1(\R_+; \dot B_{2,1}^{\frac{d}{2}-1}), \  \tilde{c}^{\ l^+,\varepsilon'}\in \mathcal{C}_b(\R_+:\dot B_{2,1}^{\frac{d}{2}}), \\ \tilde{c}^{\ l^+,\varepsilon'}\in L^1(\R_+;\dot B_{2,1}^{\frac{d}{2}+2}), \ v^l\in \mathcal{C}_b(\R_+;\dot B_{2,1}^{\frac{d}{2}}), \varepsilon v^{l^-,\varepsilon'}\in L^1(\R_+;\dot B_{2,1}^{\frac{d}{2}}), v^{l^+,\varepsilon'}\in L^1(\R_+;\dot B_{2,1}^{\frac{d}{2}+1}), \\ (\tilde{c},v)^h\in\mathcal{C}_b(\R_+;\dot B_{2,1}^{\frac{d}{2}+1})\cap L^1(\R_+,\dot B_{2,1}^{\frac{d}{2}+1}), \ W^l\in\mathcal{C}_b(\R_+;\dot B_{2,1}^{\frac{d}{2}})\cap L^1(\R_+;\dot B_{2,1}^{\frac{d}{2}})\bigg\}
\end{multline*} where we denote $\displaystyle W\mathrel{\mathop:}=-\partial_t v$.

~
\\ 
Moreover, we have the following inequality : $$\mathcal{Z}(t)\leq C \mathcal{Z}_0^{\varepsilon'}$$ where \begin{multline*}
    \mathcal{Z}(t)\mathrel{\mathop:}= \displaystyle\|\varepsilon\tilde{c}\|_{L^\infty\big(\dot B_{2,1}^{\frac{d}{2}-1}\big)}^{l^-,\varepsilon'}+\|\tilde{c}\|_{L^\infty\big(\dot B_{2,1}^{\frac{d}{2}}\big)}^{l^+,\varepsilon'}+\|v\|_{L^\infty \big(\dot B_{2,1}^{\frac{d}{2}}\big)}^{l}+\|(\tilde{c},v)\|_{L^\infty\big(\dot B_{2,1}^{\frac{d}{2}+1}\big)}^h  + \varepsilon^2 \|\varepsilon\tilde{c}\|_{L^1\big(\dot B_{2,1}^{\frac{d}{2}-1}\big)}^{l^-,\varepsilon'}+\varepsilon\|v\|_{L^1\big(\dot B_{2,1}^{\frac{d}{2}}\big)}^{l^-,\varepsilon'} \\+\|\tilde{c}\|_{L^1\big(\dot B_{2,1}^{\frac{d}{2}+2}\big)}^{l^+,\varepsilon'}+\|v\|_{L^1\big(\dot B_{2,1}^{\frac{d}{2}+1}\big)}^{l^+,\varepsilon'}+\|(\tilde{c},v)\|_{L^1\big(\dot B_{2,1}^{\frac{d}{2}+1}\big)}^h +\|W\|_{L^\infty\big(\dot B_{2,1}^{\frac{d}{2}}\big)}+\|W\|_{L^1\big(\dot B_{2,1}^{\frac{d}{2}}\big)}.
\end{multline*}
\end{theorem}

The first step is to approximate \eqref{système avec G}.

~\\
\underline{\textbf{{\emph{(1)} Approximate systems}}}~\\
Let us take $J_n$ the spectral truncation operator 
on $\displaystyle \{\xi\in\R^d,\;  n^{-1}\leq |\xi|\leq n\}$. 

We consider the following system : $$\frac{d}{dt}\begin{pmatrix} \tilde{c}\\ v \end{pmatrix}+\begin{pmatrix} J_n\left(J_n(v)\cdot \left(\nabla J_n(c)\right)\right)+\tilde{\gamma} J_n\left(J_n(c)\dive\left(J_n(v)\right)\right) \\ J_n\left(J_n(v)\cdot \nabla \left(J_n(v)\right)\right)+\tilde{\gamma}J_n\left(J_n(c)\nabla\left(J_n(c)\right)\right)+J_n(v) \end{pmatrix}=\begin{pmatrix} 0 \\ -\varepsilon^2 \nabla (-\Delta)^{-1} J_n(G(J_n(\tilde{c}))) \end{pmatrix}.$$ 
\begin{itemize}
    \item[$\bullet$] By the Cauchy-Lipschitz theorem, we have (using the spectral truncation operator) that this system admits a maximal solution $(c_n,v_n)\in \mathcal{C}^1([0,T^n[: L^2)$ with initial data ($J_n c_0, J_n v_0)$ for all $n\in\N$.
\item[$\bullet$] We have $J_n c_n=c_n$ and  $J_n(v_n)=v_n$ (by using the uniqueness in the previous system) and thus:  $$\left\{ \begin{array}{ll} \partial_t c_n + J_n(v_n\cdot \nabla c_n) +\tilde{\gamma} J_n(c_n \dive(v_n))=0, \\ \partial_t v_n +J_n(v_n\cdot \nabla v_n) + \tilde{\gamma} J_n(c_n \nabla c_n) + v_n=- \varepsilon^2\nabla \left(-\Delta\right)^{-1}J_n\left(G(\tilde{c_n})\right). \end{array} \right. $$
\item[$\bullet$] From the lemma \ref{estimées finales}, we deduce (the $n$ index corresponding to the sequence $(c_n,v_n)$): $$\mathcal{L}^n(t)+\int_0^t H^n(\tau)d\tau \lesssim \mathcal{L}^n(0)\leq \mathcal{L}(0).$$
In particular (by argument of extension of the maximal solution), we have that $T^n=+\infty$. 
\end{itemize}
~\\
\underline{\textbf{{\emph{(2)} Convergence of the sequence}}}~\\
The previous estimates guarantee that  $(\tilde{c}_n,v_n)_{n\in\N}$ is a bounded sequence of $E,$ where $E$ is the functional space described in the theorem.

In particular,  $(\tilde{c}_n,v_n)_{n\in\N}$  is bounded in $L^\infty(\R_+; \dot B_{2,1}^{\frac{d}{2}})\cap L^1(\R_+;\dot B_{2,1}^{\frac{d}{2}+2})$  and in $L^\infty(\R_+; \dot B_{2,1}^{\frac{d}{2}+1})\cap L^1(\R_+;\dot B_{2,1}^{\frac{d}{2}+1})$ at low and respectively high frequencies level, so bounded (by interpolation) in $L^2\left(\dot B_{2,1}^{\frac{d}{2}+1}\right)$.

We know that $\dot B_{2,1}^{\frac{d}{2}}$ is included continuously in $L^\infty$, hence $\dot B_{2,1}^{\frac{d}{2}+1}$ is locally compact in $L^2$.

We can therefore apply the Ascoli theorem and after diagonal extraction, we gather that, up to subsequence,  $(c_n,v_n)_{n\in\N}$ converges to some $(c,v)$ in $\mathcal{C}([0,T[;\mathcal{S}'(\R^d))$. 

By classical arguments of weak  compactness, one can conclude as in e.g \cite{BCD} that $(c,v)$ belongs to $E$ and that $(c,v)$ is a solution of the initial system.

\subsection{Proof of uniqueness}~\\
Let $Z_1=(c_1,v_1)$ and $Z_2=(c_2,v_2)$ be two solutions. We denote $\delta c\mathrel{\mathop:}=c_1-c_2,$ $\delta v\mathrel{\mathop:}=v_1-v_2$ and $\delta Z\mathrel{\mathop:}= Z_1-Z_2$.

In particular, we have :
$$\left\{\begin{array}{l}
     \partial_t \delta c+v_2\cdot \nabla\delta c+\tilde{\gamma}c_2\dive(\delta v)=-\delta v \nabla c_1-\tilde{\gamma}\delta c \dive(v_1) \\[1ex]
     \partial_t \delta v+v_2\cdot \nabla \delta v+\tilde{\gamma} c_2 \nabla \delta c +\varepsilon'^2 \nabla(-\Delta)^{-1}\delta c +\delta v=-\delta v \cdot \nabla v_1 -\tilde{\gamma}\delta c \nabla c_1-\varepsilon^2\nabla (-\Delta)^{-1}\left(F(\tilde{c_1})-F(\tilde{c_2})\right).
\end{array} \right.$$

\begin{lemma}
    We have the inequality : $$\displaylines{\|\delta Z(t)\|_{\dot B_{2,1}^{\frac{d}{2}}}^h+\|\delta v(t)\|_{\dot B_{2,1}^{\frac{d}{2}}}^l+ \|\delta c(t)\|_{\dot B_{2,1}^{\frac{d}{2}}}^{l^+,\varepsilon}+\varepsilon\|\delta c(t)\|_{\dot B_{2,1}^{\frac{d}{2}-1}}^{l^-,\varepsilon}\lesssim \int_0^t \left(\mathcal{L}_1+\mathcal{L}_2+\mathcal{H}_1+\mathcal{H}_2+1\right)(\tau) \bigg(\|\delta Z(t)\|_{\dot B_{2,1}^{\frac{d}{2}}}^h \hfill\cr\hfill+\|\delta v(t)\|_{\dot B_{2,1}^{\frac{d}{2}}}^l + \|\delta c(t)\|_{\dot B_{2,1}^{\frac{d}{2}}}^{l^+,\varepsilon}+\varepsilon\|\delta c(t)\|_{\dot B_{2,1}^{\frac{d}{2}-1}}^{l^-,\varepsilon}\bigg)d\tau} $$ where $\mathcal{L}_i, \mathcal{H}_i$ for $i \in\{1,2\}$ correspond to $\mathcal{L}$ and $\mathcal{H}$ in \eqref{LH}  for $Z_i$.
\end{lemma}
Once this lemma is proven, it is easy to conclude the uniqueness by Grönwall's lemma. 

\begin{proof}~
\begin{enumerate}
\item \underline{Estimate for high frequencies} : 

Let us start by proving an estimate for high frequencies.
By applying the localization operation $\dot\Delta_j$, we get : 
$$\left\{\begin{array}{l}
    \partial_t \delta c_j+v_2\cdot \nabla \delta c_j+\tilde{\gamma}c_2\dive(\delta v_j)=\left[v_2\cdot \nabla,\dot \Delta_j\right]\delta c  +\left[\tilde{\gamma}c_2,\dot\Delta_j\right]\dive(\delta c)-\dot \Delta_j\left(\delta v\cdot \nabla c_1+\tilde{\gamma}\delta c\dive(v_1)\right) \\
\partial_t \delta v_j+v_2\cdot \nabla \delta v_j+\tilde{\gamma} c_2 \nabla \delta c_j+\varepsilon'^2 \nabla (-\Delta)^{-1}\delta c_j+\delta v_j  =\left[v_2\cdot \nabla,\dot \Delta_j\right]\delta v+\left[\tilde{\gamma}c_2\nabla,\dot\Delta_j\right]\delta c \\ \hfill -\dot\Delta_j \left(\delta v\cdot \nabla v_1+\tilde{\gamma}\delta c\nabla c_1+\varepsilon^2\nabla (-\Delta)^{-1}\left(F(\tilde{c}_1)-F(\tilde{c}_2)\right)\right)
\end{array} \right.$$


On the one hand, owing to commutator estimates (see e.g \cite{BCD}), we have: 
$$\begin{aligned}\|\left[v_2\cdot \nabla,\dot \Delta_j\right]\delta c  +\left[\tilde{\gamma}c_2,\dot\Delta_j\right]\dive(\delta c)+\left[v_2\cdot \nabla,\dot \Delta_j\right]\delta v+\left[\tilde{\gamma}c_2\nabla,\dot\Delta_j\right]\delta c\|_{\dot B_{2,1}^{\frac{d}{2}}}^h&\lesssim \|Z_2\|_{\dot B_{2,1}^{\frac{d}{2}+1}}\|\delta Z\|_{\dot B_{2,1}^{\frac{d}{2}}}^h \\&\lesssim \mathcal{H}_2(\tau) \|\delta Z\|_{\dot B_{2,1}^{\frac{d}{2}}}^h.\end{aligned}$$

On the other hand, we have by integration by parts:
\begin{itemize}
    \item[$\bullet$] $\displaystyle -\int_{\R^d} v_2\cdot \nabla \delta c_j \delta c_j dx=\frac{1}{2}\int_{\R^d}\dive(v_2)|\delta c_j|^2 dx$,
    \item[$\bullet$] $\displaystyle -\int_{\R^d} v_2\cdot \nabla \delta v_j \cdot\delta v_j dx=\frac{1}{2}\int_{\R^d}\dive(v_2)|\delta v_j|^2 dx$,
    \item[$\bullet$] $\displaystyle -\int_{\R^d}\tilde{\gamma}\tilde{c}_2 \dive(\delta v_j)\delta c_j dx-\int_{\R^d} \tilde{\gamma}\tilde{c}_2 \nabla \delta c_j\cdot \delta v_j dx=-\int_{\R^d}\tilde{\gamma}\tilde{c}_2\dive(\delta c_j \delta v_j)dx=\int_{\R^d}\tilde{\gamma} \nabla c_2 \cdot \left(\delta c_j \delta v_j\right)dx.$
\end{itemize}
Then we have for all $j\leq -1$ : \begin{align*}
    -\int_{\R^d} v_2\cdot \nabla \delta c_j \delta c_j dx-\int_{\R^d} v_2\cdot \nabla \delta v_j \cdot\delta v_j dx &-\int_{\R^d}\tilde{\gamma}\tilde{c}_2 \dive(\delta v_j)\delta c_j dx-\int_{\R^d} \tilde{\gamma}\tilde{c}_2 \nabla \delta c_j\cdot \delta v_j dx \\
&=\frac{1}{2}\int_{\R^d} \dive(v_2) |\delta Z_j|^2 dx+\int_{\R^d} \tilde{\gamma}\nabla c_2\cdot(\delta c_j\delta v_j)dx \\ &\lesssim  a_j 2^{-j\frac{d}{2}}\|Z_2\|_{\dot B_{2,1}^{\frac{d}{2}+1}}\|\delta Z\|_{\dot B_{2,1}^{\frac{d}{2}}}\|\delta Z\|_{\dot B_{2,1}^{\frac{d}{2}}} \\ &\lesssim  a_j 2^{-j\frac{d}{2}}\mathcal{H}_2(t)\|\delta Z\|_{\dot B_{2,1}^{\frac{d}{2}}}\|\delta Z\|_{\dot B_{2,1}^{\frac{d}{2}}}.\end{align*}

We have also : $$\begin{aligned}\|-\dot \Delta_j\left(\delta v\cdot \nabla c_1+\tilde{\gamma}\delta c\dive(v_1)\right)&-\dot\Delta_j \left(\delta v\cdot \nabla v_1+\tilde{\gamma}\delta c\nabla c_1+\varepsilon^2\nabla (-\Delta)^{-1}\left(F(\tilde{c}_1)-F(\tilde{c}_2)\right)\right)\|_{L^2} \\&\lesssim a_j 2^{-j\frac{d}{2}}\|\delta Z\|_{\dot B_{2,1}^{\frac{d}{2}}}\|Z_1\|_{\dot B_{2,1}^{\frac{d}{2}+1}}\\&\lesssim a_j 2^{-j\frac{d}{2}}\|\delta Z\|_{\dot B_{2,1}^{\frac{d}{2}}} \mathcal{H}_1.\end{aligned}$$

We then have (taking the scalar product with $\delta Z_j$ in the previous system) the following estimate: 
$$\|\delta Z\|_{\dot B_{2,1}^{\frac{d}{2}}}^h+\int_0^t \|\delta Z\|_{\dot B_{2,1}^{\frac{d}{2}}}^hd\tau\lesssim \int_0^t (\mathcal{H}_1+\mathcal{H}_2)(\tau)\|\delta Z\|_{\dot B_{2,1}^{\frac{d}{2}}}^h d\tau.$$

\item \underline{Estimates for low frequencies} :

As for the study of the Euler-Poisson system, we will look at $\delta v$ in $\dot B_{2,1}^{\frac{d}{2}}$, the very low frequencies of $\varepsilon \delta c$ in $\dot B_{2,1}^{\frac{d}{2}-1}$ and the medium ones of $\delta c$ in $\dot B_{2,1}^{\frac{d}{2}}$.

For the above system, we obtain the following estimates:
$$\left\{\begin{array}{l}
\|\varepsilon\delta c(t)\|_{\dot B_{2,1}^{\frac{d}{2}-1}}^{l^-,\varepsilon}\lesssim \displaystyle \int_0^t \varepsilon\left[\|\delta v\|_{\dot B_{2,1}^{\frac{d}{2}}}^{l^-,\varepsilon}+\|(Z_1,Z_2)\|_{\dot B_{2,1}^{\frac{d}{2}}}\|\delta Z\|_{\dot B_{2,1}^{\frac{d}{2}}} \right]d\tau
\\
     \|\delta c(t)\|_{\dot B_{2,1}^{\frac{d}{2}}}^{l^+,\varepsilon} \lesssim \displaystyle\int_0^t \left[\|\delta v\|_{\dot B_{2,1}^{\frac{d}{2}+1}}^{l^+,\varepsilon}+ \|Z_2\|_{\dot B_{2,1}^{\frac{d}{2}}}\left(\|\delta Z\|_{\dot B_{2,1}^{\frac{d}{2}+1}}^{l}+\|\delta Z\|_{\dot B_{2,1}^{\frac{d}{2}}}^{h}\right)+\|\delta Z\|_{\dot B_{2,1}^{\frac{d}{2}}}\|Z_1\|_{\dot B_{2,1}^{\frac{d}{2}+1}}\right]d\tau
\\ \displaystyle
\|\delta v\|_{\dot B_{2,1}^{\frac{d}{2}+1}}^{l}+\int_0^t \|\delta v\|_{\dot B_{2,1}^{\frac{d}{2}+1}}^{l} d\tau\lesssim \int_0^t \bigg[\varepsilon^2 \|\delta c\|_{\dot B_{2,1}^{\frac{d}{2}-1}}^{l^-,\varepsilon}+\|\delta c\|_{\dot B_{2,1}^{\frac{d}{2}+1}}^{l^+,\varepsilon}+\|Z_2\|_{\dot B_{2,1}^{\frac{d}{2}}}\left(\|\delta Z\|_{\dot B_{2,1}^{\frac{d}{2}+1}}^{l}+\|\delta Z\|_{\dot B_{2,1}^{\frac{d}{2}}}^{h}\right) \cr\hfill +\|\delta Z\|_{\dot B_{2,1}^{\frac{d}{2}}}\|Z_1\|_{\dot B_{2,1}^{\frac{d}{2}+1}}+\varepsilon^2\big(\|\delta c\|_{\dot B_{2,1}^{\frac{d}{2}}}\|(c_1,c_2)\|_{\dot B_{2,1}^{\frac{d}{2}-1}} \cr\hfill +\|\delta c\|_{\dot B_{2,1}^{\frac{d}{2}-1}}\|(c_1,c_2)\|_{\dot B_{2,1}^{\frac{d}{2}}}\big)\bigg]d\tau.
\end{array} \right.$$

\item \underline{Final estimate} :
We now put together all the estimates and observe that some terms in the right-hand side are negligible compared to the 
left-hand side (thanks in particular to our lemma about \ref{estimées finales} a priori estimates) and we obtain the final result.
\end{enumerate}
\end{proof}

We deduce by change of variable \eqref{changement de variable 2} and lemma \ref{dilatation} the statement of the theorem \ref{théorème Euler-Poisson}.

\section{Keller-Segel parabolic/elliptical system}
The goal of this section is to justify the convergence of the density solution of the first equation of \eqref{système avant passage à la limite} to the unique solution of \eqref{Keller-Segel} when $\varepsilon$ tends to 0.

We deduce from the theorem \ref{théorème Euler-Poisson} the following theorem which will allow us to study the singular limit of the Euler-Poisson system:
\begin{theorem}\label{dernier Euler-Poisson} Let be $\varepsilon>0$. Let $\varepsilon'$ be defined as \eqref{epsilon'}.
There exists a positive constant $\alpha$ such that for all $\displaystyle\varepsilon'\leq \frac{1}{2}$ and data $Z_0^\varepsilon=(\varrho_0^\varepsilon-\overline{\varrho},v_0^\varepsilon)\left(\dot B_{2,1}^{\frac{d}{2}-1}\cap \dot B_{2,1}^{\frac{d}{2}+1}\right)\times \left(\dot B_{2,1}^{\frac{d}{2}}\cap \dot B_{2,1}^{\frac{d}{2}+1}\right)  $ satisfying : 

$$\mathcal{Z}_0^{\varepsilon}\mathrel{\mathop:}= \|\varrho_0^\varepsilon-\overline{\varrho}\|_{\dot B_{2,1}^{\frac{d}{2}-1}}^{l^-, \ \varepsilon' \varepsilon^{-1}}+ \|\varrho_0^\varepsilon-\overline{\varrho}\|_{\dot B_{2,1}^{\frac{d}{2}}}^{l^+,\ \varepsilon' \varepsilon^{-1}, \ \varepsilon}+\varepsilon\| v_0^\varepsilon\|_{\dot B_{2,1}^{\frac{d}{2}}}^{l,\ \varepsilon^{-1}}+\varepsilon\|\left(\varrho_0^\varepsilon-\overline{\varrho},\varepsilon v_0^\varepsilon\right)\|_{\dot B_{2,1}^{\frac{d}{2}+1}}^{h,\ \varepsilon^{-1}}\leq \alpha, $$ the system \eqref{Euler-Poisson initial} with the initial data $(\varrho_0^\varepsilon,v_0^\varepsilon)$ admits a unique global-in-time solution $Z^\varepsilon=(\varrho^\varepsilon-\bar\rho,v^\varepsilon)$ in the set 
\begin{multline*}
\tilde{E}\mathrel{\mathop:}=\bigg\{(\varrho^\varepsilon-\overline{\varrho},v^\varepsilon) \ \bigg| \  (\varrho^\varepsilon-\overline{\varrho})^{l^-,\ \varepsilon'\varepsilon^{-1}}\in \mathcal{C}_b(\R_+:\dot B_{2,1}^{\frac{d}{2}-1}), \ \varepsilon(\varrho^\varepsilon-\overline{\varrho})^{l^-,\ \varepsilon'\varepsilon^{-1}}\in L^1(\R_+; \dot B_{2,1}^{\frac{d}{2}-1}), \\ \ (\varrho^\varepsilon-\overline{\varrho})^{l^+, \ \varepsilon'\varepsilon^{-1}, \ \varepsilon^{-1}}\in \mathcal{C}_b(\R_+:\dot B_{2,1}^{\frac{d}{2}}), \ 
 \varepsilon(\varrho^\varepsilon-\overline{\varrho})^{l^+, \ \varepsilon'\varepsilon^{-1}}\in L^1(\R_+;\dot B_{2,1}^{\frac{d}{2}+1}), \ (\varrho^\varepsilon-\overline{\varrho})^{l^+, \ \varepsilon'\varepsilon^{-1}, \ \varepsilon^{-1}}\in L^1(\R_+;\dot B_{2,1}^{\frac{d}{2}+1}), \\  \varepsilon(v^\varepsilon) ^{l, \ \varepsilon^{-1}}\in \mathcal{C}_b(\R_+;\dot B_{2,1}^{\frac{d}{2}}), \ (v^\varepsilon)^{l^-, \ \varepsilon'\varepsilon^{-1}}\in L^1(\R_+;\dot B_{2,1}^{\frac{d}{2}}), \ (v^\varepsilon)^{l^+, \ \varepsilon'\varepsilon^{-1}}\in L^1(\R_+;\dot B_{2,1}^{\frac{d}{2}+1}) , \\ (\varrho^\varepsilon-\overline{\varrho},v^\varepsilon)^{h, \varepsilon^{-1}}\in\mathcal{C}_b(\R_+;\dot B_{2,1}^{\frac{d}{2}-1})\cap L^1(\R_+;\dot B_{2,1}^{\frac{d}{2}+1}), \ w^\varepsilon\in \mathcal{C}_b(\R_+;\dot B_{2,1}^{\frac{d}{2}})\cap L^1(\R_+;\dot B_{2,1}^{\frac{d}{2}}) \bigg\}
\end{multline*}

where we have denoted $\displaystyle w^\varepsilon\mathrel{\mathop:}=\varepsilon\frac{\nabla\left(P(\varrho^\varepsilon)\right)}{\varrho^\varepsilon}+ v^\varepsilon+\varepsilon\nabla (-\Delta)^{-1}(\varrho^\varepsilon-\overline{\varrho})$.

~

Moreover, we have the following inequality : $$\mathcal{Z}^\varepsilon(t)\leq C \mathcal{Z}_0^\varepsilon$$ where \begin{multline*}
    \mathcal{Z}^\varepsilon(t)\mathrel{\mathop:}= \displaystyle\|\varrho^\varepsilon-\overline{\varrho}\|_{L^\infty\left(\dot B_{2,1}^{\frac{d}{2}-1}\right)}^{l^-,\ \varepsilon'\varepsilon^{-1}}+\|\varrho^\varepsilon-\overline{\varrho}\|_{L^\infty\left(\dot B_{2,1}^{\frac{d}{2}}\right)}^{l^+,\ \varepsilon'\varepsilon^{-1}, \ \varepsilon^{-1} }+\varepsilon\|v^\varepsilon\|_{L^\infty \left(\dot B_{2,1}^{\frac{d}{2}}\right)}^{l, \  \varepsilon^{-1}}+\varepsilon\|(\varrho^\varepsilon-\overline{\varrho},\varepsilon v^\varepsilon)\|_{L^\infty\left(\dot B_{2,1}^{\frac{d}{2}+1}\right)}^{h, \ \varepsilon^{-1}}
    \\
   + \|\varrho^\varepsilon-\overline{\varrho}\|_{L^1\left(\dot B_{2,1}^{\frac{d}{2}-1}\right)}^{l^-,\ \varepsilon'\varepsilon^{-1}}+\|v^\varepsilon\|_{L^1\left(\dot B_{2,1}^{\frac{d}{2}}\right)}^{l^-,\ \varepsilon'\varepsilon^{-1}}+\|\varrho^\varepsilon-\overline{\varrho}\|_{L^1\left(\dot B_{2,1}^{\frac{d}{2}+2}\right)}^{l^+, \ \varepsilon}+\|v^\varepsilon\|_{L^1\left(\dot B_{2,1}^{\frac{d}{2}+1}\right)}^{l^+,\ \varepsilon}\\ +\varepsilon^{-1}\|(\varrho^\varepsilon-\overline{\varrho},\varepsilon v^\varepsilon)\|_{L^1\left(\dot B_{2,1}^{\frac{d}{2}+1}\right)}^{h, \varepsilon^{-1}} +\|w^\varepsilon\|_{L^\infty\left(\dot B_{2,1}^{\frac{d}{2}}\right)}+\varepsilon^{-2}\|w^\varepsilon\|_{L^1\left(\dot B_{2,1}^{\frac{d}{2}}\right)}.
\end{multline*}
\end{theorem}

We notice with theorem \ref{dernier Euler-Poisson} ensures that $\displaystyle \tilde{W}^\varepsilon= \mathcal{O}(\varepsilon)$ in $\displaystyle L^1(\R_+; \dot B_{2,1}^{\frac{d}{2}})$.
As the first equation of \eqref{système avant passage à la limite} can be rewritten as $$\partial_t \tilde{\varrho}^\varepsilon-\Delta\left(P(\tilde{\varrho}^\varepsilon)\right)-\dive\left(\tilde{\varrho}^\varepsilon \nabla (-\Delta)^{-1}(\tilde{\varrho}^\varepsilon-\overline{\varrho})\right)=\dive(\tilde{\varrho}^\varepsilon \tilde{W}^\varepsilon),$$ we suspect that the density will tend to satisfy the  parabolic-elliptic Keller-Segel system \eqref{Keller-Segel} supplemented with the initial data $\displaystyle \underset{\varepsilon\to 0}{\lim}\,\tilde{\varrho}_0^\varepsilon$.

Let us now rigorously prove the theorem \ref{theoreme final} :

\begin{proof}
Let us justify quickly that for all $N_0$ satisfying \eqref{condition initiale keller-segel}, there exists a unique global-in-time solution ${N}$ of \eqref{Keller-Segel} in $\mathcal{C}_b\left(\R_+;\dot{B}_{2,1}^{\frac{d}{2}-1}\cap\dot B_{2,1}^\frac{d}{2}\right)\cap L^1\left(\R_+;\dot{B}_{2,1}^{\frac{d}{2}+2}\cap\dot B_{2,1}^{\frac{d}{2}-1}\right)$ satisfying \eqref{estimation keller-segel}.

In terms of  $\tilde{N}:=N-\overline{\varrho},$  the equation \eqref{Keller-Segel} is rewritten : $$\partial_t \tilde{N}-\Delta(P(N))-\dive\left(N \nabla (-\Delta)^{-1}\tilde{N}\right)=0.$$  

By the Taylor-Lagrange formula, we notice that : $$P(N)-P(\overline{\varrho})=\tilde{N}P'(\overline{\varrho})+g(\tilde{N})$$ where $g$ is a smooth function vanishing at $0$ (and also its first derivative). 

\begin{enumerate}
    \item \underline{Low frequency analysis}
    ~\\
    We can rewrite this system as : $$\partial_t \tilde{N}+\overline{\varrho}\tilde{N}=
    \Delta\bigl(\tilde{N}P'(\overline{\varrho})+g(\tilde{N})\bigr)+\dive\left(\tilde{N}\nabla(-\Delta)^{-1} \tilde{N}\right)\cdotp$$
    
    We then obtain the following estimates: 
    $$\|\tilde{N}(t)\|_{\dot B_{2,1}^{\frac{d}{2}-1}}^l+\int_0^t \|\tilde{N}\|_{\dot B_{2,1}^{\frac{d}{2}-1}}^l d\tau \lesssim \|\tilde{N}_0\|_{\dot B_{2,1}^{\frac{d}{2}-1}}^l+\int_0^t \|g(\tilde{N})\|_{\dot B_{2,1}^{\frac{d}{2}+1}}d\tau+\int_0^t \|\tilde{N}\|_{\dot B_{2,1}^{\frac{d}{2}+1}}^l d\tau +\int_0^t \|\tilde{N}\|_{\dot B_{2,1}^{\frac{d}{2}}}\|\tilde{N}\|_{\dot B_{2,1}^{\frac{d}{2}-1}}d\tau,$$ 
    We note that the term $\displaystyle \int_0^t \|\tilde{N}\|_{\dot B_{2,1}^{\frac{d}{2}+1}}^l d\tau$ is negligible compared to $\displaystyle \int_0^t \|\tilde{N}\|_{\dot B_{2,1}^{\frac{d}{2}-1}}^l d\tau.$

    We deduce that : 
    $$\|\tilde{N}(t)\|_{\dot B_{2,1}^{\frac{d}{2}-1}}^l+\int_0^t \|\tilde{N}\|_{\dot B_{2,1}^{\frac{d}{2}-1}}^l d\tau \lesssim \|\tilde{N}_0\|_{\dot B_{2,1}^{\frac{d}{2}-1}}^l+ \int_0^t \|\tilde{N}\|_{\dot B_{2,1}^{\frac{d}{2}-1}\cap\dot B_{2,1}^{\frac{d}{2}}} \|\tilde{N}\|_{\dot B_{2,1}^{\frac{d}{2}+2}\cap \dot B_{2,1}^{\frac{d}{2}-1}} d\tau. $$
    \item \underline{High frequency analysis}
    ~\\
    We can rewrite the system as : $$\partial_t \tilde{N}-P'(\overline{\varrho})\Delta \tilde{N}=\Delta(g(\tilde{N}))-\dive(N\nabla (-\Delta)^{-1} \tilde{N}).$$
    
    Then, we get : 
    $$\displaylines{\|\tilde{N}(t)\|_{\dot B_{2,1}^{\frac{d}{2}}}^h+\int_0^t \|\tilde{N}\|_{\dot B_{2,1}^{\frac{d}{2}+2}}^hd\tau \lesssim \|\tilde{N}_0\|_{\dot B_{2,1}^{\frac{d}{2}}}^h+\int_0^t \|\tilde{N}\|_{\dot B_{2,1}^{\frac{d}{2}+2}}\|\tilde{N}\|_{\dot B_{2,1}^{\frac{d}{2}}}d\tau+\int_0^t \|\tilde{N}\|_{\dot B_{2,1}^{\frac{d}{2}}}^h d\tau \hfill\cr\hfill+\int_0^t \left( \|\tilde{N}\|_{\dot B_{2,1}^{\frac{d}{2}+1}}\|\tilde{N}\|_{\dot B_{2,1}^{\frac{d}{2}-1}}+ \|\tilde{N}\|_{\dot B_{2,1}^{\frac{d}{2}}}^2\right)d\tau.}$$

    Thus, we have : $$\|\tilde{N}(t)\|_{\dot B_{2,1}^{\frac{d}{2}}}^h+\int_0^t \|\tilde{N}\|_{\dot B_{2,1}^{\frac{d}{2}+2}}^h d\tau \lesssim \|\tilde{N}_0\|_{\dot B_{2,1}^{\frac{d}{2}}}^h+ \int_0^t \left(\|\tilde{N}\|_{\dot B_{2,1}^{\frac{d}{2}+2}}\|\tilde{N}\|_{\dot B_{2,1}^{\frac{d}{2}}}+\|\tilde{N}\|_{\dot B_{2,1}^{\frac{d}{2}-1}} \|\tilde{N}\|_{\dot B_{2,1}^{\frac{d}{2}+1}}\right) d\tau. $$
    \item \underline{A priori estimate}
    ~\\
    By gathering the previous information, we get:    
$$\|\tilde{N}(t)\|_{\dot B_{2,1}^{\frac{d}{2}-1}\cap \dot B_{2,1}^{\frac{d}{2}}}+\int_0^t \|\tilde{N}\|_{\dot B_{2,1}^{\frac{d}{2}+2}\cap \dot B_{2,1}^{\frac{d}{2}-1}} d\tau \lesssim \|\tilde{N}_0\|_{\dot B_{2,1}^{\frac{d}{2}-1}\cap \dot B_{2,1}^{\frac{d}{2}}}+ \int_0^t \|\tilde{N}\|_{\dot B_{2,1}^{\frac{d}{2}-1}\cap \dot B_{2,1}^{\frac{d}{2}}} \|\tilde{N}\|_{\dot B_{2,1}^{\frac{d}{2}+2}\cap \dot B_{2,1}^{\frac{d}{2}-1}} d\tau. $$
    Then, we have: $$\|\tilde{N}(t)\|_{\dot B_{2,1}^{\frac{d}{2}-1}\cap \dot B_{2,1}^{\frac{d}{2}}}+\int_0^t \|\tilde{N}\|_{\dot B_{2,1}^{\frac{d}{2}+2}\cap \dot B_{2,1}^{\frac{d}{2}-1}} d\tau \lesssim \|\tilde{N}_0\|_{\dot B_{2,1}^{\frac{d}{2}-1}\cap \dot B_{2,1}^{\frac{d}{2}}}.$$
    Hence \eqref{estimation keller-segel}.
    \\
    By taking advantage of the Picard fixed point theorem in the functional framework given by the inequalities above, we obtain an unique global-in-time solution $\tilde{N}$ of \eqref{Keller-Segel} in $\mathcal{C}_b\Bigl(\R_+;\dot{B}_{2,1}^{\frac{d}{2}-1}\cap\dot B_{2,1}^\frac{d}{2}\Bigr)\cap L^1\Bigl(\R_+;\dot{B}_{2,1}^{\frac{d}{2}+2}\cap\dot B_{2,1}^{\frac{d}{2}-1}\Bigr)$ satisfying \eqref{estimation keller-segel}.
\end{enumerate}

In order  to prove the last part of this theorem, let us observe that $(N,\tilde{\varrho}^\varepsilon)$ satisfies : $$\left\{\begin{array}{l} \partial_t \tilde{\varrho}^\varepsilon-\Delta(P(\tilde{\varrho}^\varepsilon))-\dive\left(\tilde{\varrho}^\varepsilon \nabla(-\Delta)^{-1} (\tilde{\varrho}^\varepsilon-\overline{\varrho}\right)=\dive\left(\tilde{\varrho}^\varepsilon \tilde{W}^\varepsilon\right),
\\
\partial_t N-\Delta(P(N))-\dive\left(N\nabla (-\Delta)^{-1} (N-\overline{\varrho})\right)=0 .    
\end{array} \right.$$

Let us denote $\delta N=N-\tilde{\varrho}^\varepsilon$. We obtain : $$\displaylines{\partial_t \delta N+\Delta \left(P(\tilde{\varrho}^\varepsilon\right)-\Delta\left(P(N)\right)+\overline{\varrho} \ \delta N-\dive\left(\delta N \nabla (-\Delta)^{-1} (N-\overline{\varrho})\right)\hfill\cr\hfill-\dive\left((\tilde{\varrho}^\varepsilon-\overline{\varrho})\nabla (-\Delta)^{-1} \delta N\right) =-\dive\left(\tilde{\varrho}^\varepsilon \tilde{W}^\varepsilon\right)\cdotp}$$

Let us look estimate $\delta N$ at the level of regularity $\dot B_{2,1}^{\frac{d}{2}-1}$. We have : 
$$\int_0^t \|\dive(\tilde{\varrho}^\varepsilon \tilde{W}^\varepsilon)\|_{\dot B_{2,1}^{\frac{d}{2}-1}}d\tau   \lesssim \int_0^t\|\tilde{\varrho}^\varepsilon \tilde{W}^\varepsilon\|_{\dot B_{2,1}^{\frac{d}{2}}}d\tau \lesssim \|\tilde{\varrho}^\varepsilon\|_{L_t^\infty\left(\dot B_{2,1}^{\frac{d}{2}}\right)} \int_0^t \|\tilde{W}^\varepsilon\|_{\dot B_{2,1}^{\frac{d}{2}}}d\tau \lesssim \alpha \varepsilon,$$
$$\begin{aligned}
\int_0^t \|\dive\left(\delta N \nabla (-\Delta)^{-1}(N-\overline{\varrho}\right)\|_{\dot B_{2,1}^{\frac{d}{2}-1}}d\tau & \lesssim  \int_0^t \|\delta N \nabla (-\Delta)^{-1}(N-\overline{\varrho})\|_{\dot B_{2,1}^{\frac{d}{2}}}d\tau \cr & \lesssim \int_0^t \|\delta N\|_{\dot B_{2,1}^{\frac{d}{2}}} \|(N-\overline{\varrho})\|_{\dot B_{2,1}^{\frac{d}{2}-1}}d\tau \cr &\lesssim \alpha \int_0^t \|\delta N\|_{\dot B_{2,1}^{\frac{d}{2}}} d\tau,\end{aligned}$$ 
$$\begin{aligned}
 \int_0^t \|\dive\left((\tilde{\varrho}^\varepsilon-\overline{\varrho}) \nabla(-\Delta)^{-1}\delta N\right)\|_{\dot B_{2,1}^{\frac{d}{2}-1}}d\tau & \lesssim \int_0^t \|(\tilde{\varrho}^\varepsilon-\overline{\varrho}) \nabla(-\Delta)^{-1}\delta N\|_{\dot B_{2,1}^{\frac{d}{2}}}d\tau \cr & \lesssim \|\tilde{\varrho}^\varepsilon-\overline{\varrho}\|_{L^\infty(\dot B_{2,1}^{\frac{d}{2}})}\int_0^t \|\delta N\|_{\dot B_{2,1}^{\frac{d}{2}-1}}d\tau \cr & \lesssim \alpha \int_0^t \|\delta N\|_{\dot B_{2,1}^{\frac{d}{2}-1}}d\tau.
 \end{aligned}$$

In order to study $\displaystyle \Delta\left(P(\tilde{\varrho}^\varepsilon)-P(N)\right)$, we use the identity : $$\Delta \left(P(\phi)\right)= \Delta \phi P'(\phi)+|\nabla \phi|^2 P''(\phi).$$

Hence, we have : \begin{multline*}
\Delta\left(P(\tilde{\varrho}^\varepsilon)-P(N)\right)=P'(\overline{\varrho}) \Delta (\delta N)+(P'(N)-P'(\varrho))\Delta(\delta N)+ \Delta \tilde{\varrho}^\varepsilon\left(P'(N)-P'(\tilde{\varrho}^\varepsilon)\right)
\\
+\left(|\nabla N|^2-|\nabla \tilde{\varrho}^\varepsilon|^2\right)P''(N)+|\nabla \tilde{\varrho}^\varepsilon|^2\left(P''(N)-P''(\tilde{\varrho}^\varepsilon)\right)\cdotp 
\end{multline*}
Let bound the r.h.s in $L^1(\R_+;\dot B_{2,1}^{\frac{d}{2}-1})$ (in particular, the estimate of Theorem \ref{dernier Euler-Poisson} will be used but also the different lemmas in the appendix) :
$$
    \int_0^t\|\Delta(\delta N)\left(P'(N)-P'(\overline{\varrho})\right)\|_{\dot B_{2,1}^{\frac{d}{2}-1}}d\tau\lesssim \int_0^t\|\delta N\|_{\dot B_{2,1}^{\frac{d}{2}+1}}\|N-\overline{\varrho}\|_{\dot B_{2,1}^{\frac{d}{2}}}d\tau\lesssim \alpha \int_0^t\|\delta N\|_{\dot B_{2,1}^{\frac{d}{2}+1}}d\tau,$$
    $$\begin{aligned}
  \int_0^t\|\Delta \tilde{\varrho}^\varepsilon\left(P'(N)-P'(\tilde{\varrho}^\varepsilon)\right)\|_{\dot B_{2,1}^{\frac{d}{2}-1}}d\tau & \lesssim \int_0^t\|\tilde{\varrho}^\varepsilon-\overline{\varrho}\|_{\dot B_{2,1}^{\frac{d}{2}+1}} \|\delta N\|_{\dot B_{2,1}^{\frac{d}{2}}} \|(N,\tilde{\varrho}^\varepsilon)\|_{\dot B_{2,1}^{\frac{d}{2}}}d\tau\cr & \lesssim (\alpha^2+\alpha) \|\delta N\|_{L^\infty(\dot B_{2,1}^{\frac{d}{2}})},\end{aligned}$$
  $$\begin{aligned}
    \int_0^t\| |\nabla \tilde{\varrho}^\varepsilon|^2 \left(P''(N)-P''(\overline{\varrho})\right)\|_{\dot B_{2,1}^{\frac{d}{2}-1}}d\tau & \lesssim \displaystyle\int_0^t \|\tilde{\varrho}^\varepsilon\|_{\dot B_{2,1}^{\frac{d}{2}+1}} \|\tilde{\varrho}^\varepsilon\|_{\dot B_{2,1}^{\frac{d}{2}}}\|\delta N\|_{\dot B_{2,1}^{\frac{d}{2}}}\|(\tilde{N},\tilde{\varrho}^\varepsilon)\|_{\dot B_{2,1}^{\frac{d}{2}}}d\tau \cr &  \lesssim (\alpha^3+\alpha^2)  \int_0^t\|\delta N\|_{\dot B_{2,1}^{\frac{d}{2}}}d\tau,\end{aligned}$$
    $$\begin{aligned}
    \int_0^t \|\left(|\nabla N|^2-|\nabla \tilde{\varrho}^\varepsilon|^2\right)P''(N)\|_{\dot B_{2,1}^{\frac{d}{2}-1}}d\tau & \lesssim \displaystyle\int_0^t \|\nabla N-\nabla \tilde{\varrho}^\varepsilon\|_{\dot B_{2,1}^{\frac{d}{2}}}\|(\nabla N,\nabla \tilde{\varrho}^\varepsilon)\|_{\dot B_{2,1}^{\frac{d}{2}-1}}\|N\|_{\dot B_{2,1}^{\frac{d}{2}}}d\tau \cr & \lesssim (\alpha^2+\alpha) \int_0^t \|\delta N\|_{\dot B_{2,1}^{\frac{d}{2}+1}} d\tau.
\end{aligned}$$

In particular, $\delta N$ satisfies : $$\displaylines{\partial_t \delta N-P'(\overline{\varrho})\Delta \delta N+ \overline{\varrho}\delta N= -\dive\left(\tilde{\varrho}^\varepsilon\tilde{W}^\varepsilon\right)+\dive\left(\delta N \nabla (-\Delta)^{-1}(N-\overline{\varrho}\right)+\dive\left((\tilde{\varrho}^\varepsilon-\overline{\varrho}) \nabla(-\Delta)^{-1}\delta N\right) \hfill\cr\hfill-\Delta(\delta N)(P'(N)-P'(\overline{\varrho}))-\Delta \tilde{\varrho}^\varepsilon\left(P'(N)-P'(\tilde{\varrho}^\varepsilon)\right)\hfill\cr\hfill-|\nabla \tilde{\varrho}^\varepsilon|^2 \left(P''(N)-P''(\overline{\varrho})\right)-\left(|\nabla N|^2-|\nabla \tilde{\varrho}^\varepsilon|^2\right)P''(N).}$$ 

By using all previous inequalities, we get : 
$$\displaylines{\|\delta N\|_{L_t^\infty(R_+;\dot B_{2,1}^{\frac{d}{2}-1})\cap L^1(\R_+;\dot B_{2,1}^{\frac{d}{2}-1})\cap L^1(\R_+;\dot B_{2,1}^{\frac{d}{2}+1})}\lesssim \|\delta N(0)\|_{\dot B_{2,1}^{\frac{d}{2}-2}}+\alpha \varepsilon \hfill\cr\hfill +(\alpha+\alpha^2+\alpha^3)\|\delta N\|_{L_t^\infty(R_+;\dot B_{2,1}^{\frac{d}{2}-1})\cap L^1(\R_+;\dot B_{2,1}^{\frac{d}{2}-1})\cap L^1(\R_+;\dot B_{2,1}^{\frac{d}{2}+1})}.}$$

Therefore, for $\alpha$ small enough, we get : $$\|\delta N\|_{L^\infty(R_+;\dot B_{2,1}^{\frac{d}{2}-1})\cap L^1(\R_+;\dot B_{2,1}^{\frac{d}{2}-1}\cap\dot B_{2,1}^{\frac{d}{2}+1})}\lesssim \|\delta N(0)\|_{\dot B_{2,1}^{\frac{d}{2}-1}}+\alpha \varepsilon$$ which completes the proof of the theorem.
\end{proof}

\appendix
\section{}
Here we recall classic lemmas involving  differential inequalities and some basic properties on Besov spaces and product estimates have been be used repeatedly in the article.

\begin{lemma}\label{lemme edo}
Let $\displaystyle X:[0,T]\rightarrow \R_+$ be a continuous function such that $X^2$ is differentiable. Assume that there exist a constant $c\geq 0$ and a measurable function $A:[0,T]\rightarrow \R_+$ such that $$\frac{1}{2}\frac{d}{dt}X^2+c X^2\leq A X \quad  \text{a.e. on} \ [0,T].$$
Then, for all $t\in [0,T]$, we have: $$X(t)+c\int_0^t X(\tau)\,d\tau\leq X_0 +\int_0^t A(\tau) \,d\tau.$$
\end{lemma}

This classical lemma can be found for instance in \cite{RD} :
\begin{lemma}\label{lemme edo2}
    Let $T>0$. Let $\mathcal{L}:[0,T]\to \R$ and $H:[0,T]\to \R$ two continous positive functions on $[0,T]$ such that $$\mathcal{L}(t)+c\int_0^t \mathcal{H}(\tau)d\tau \leq \mathcal{L}(0)+C\int_0^t \mathcal{L}(\tau)\mathcal{H}(\tau)d\tau,$$ and $\mathcal{L}(0)\leq \alpha<<1$ then for all $t\in[0,T]$, we have : $$\mathcal{L}(t)+\frac{c}{2}\int_0^t \mathcal{H}(\tau)d\tau \leq \mathcal{L}(0).$$
\end{lemma}
\begin{proof}
    Let $\displaystyle\alpha\in ]0,\frac{c}{2C}[$. We set $T_0=\underset{T_1\in[0,T]}{\sup}\left\{\underset{t\in[0,T_1]}{\sup}\mathcal{L}(t)\leq \alpha \right\}$. This $\sup$ exists because the previous set is non empty ($0$ belongs to this set) and since $\mathcal{L}$ is continuous,  $T_0>0$. In time $t=T_0$, we get : $$\displaylines{\mathcal{L}(T_0)+c\int_0^{T_0}\mathcal{H}(\tau)d\tau \leq \mathcal{L}(0)+C\int_0^{T_0}\mathcal{H}(\tau)\mathcal{L}(\tau)d\tau\leq \mathcal{L}(0)+\alpha C\int_0^{T_0}\mathcal{H}(\tau)d\tau. 
    }$$

    Hence we have : $$\mathcal{L}(T_0)+\frac{c}{2}\int_0^{T_0} \mathcal{H}(\tau)d\tau \leq \mathcal{L}(0).$$

    As $\mathcal{L}(t)\leq \mathcal{L}(T_0)$ for all $t\in[0,T_0]$ and $\displaystyle \int_0^t \mathcal{H}(\tau) d\tau \leq \int_0^{T_0} \mathcal{H}(\tau) d\tau$, we obtain with the previous inequality : $$\mathcal{L}(t)\leq \alpha \quad \forall t\in[0,T_0].$$

    With the continuity of $\mathcal{L}$, we must have $T_0=T$, whence the result. 
\end{proof}

The following lemmas are classic results on Besov spaces ( see e.g. \cite{BCD}).
\begin{lemma} \label{Produit espace de Besov}
For all $d\geq 2$, the pointwise product extends in a continuous application of $\dot B_{2,1}^{\frac{d}{2}-1}(\R^d)\times \dot B_{2,1}^{\frac{d}{2}}(\R^d)$ to $\dot B_{2,1}^{\frac{d}{2}-1}(\R^d)$ and $ \dot B_{2,1}^{\frac{d}{2}}$ is a multiplicative algebra for all $d\geq 1$.

For all $d\geq 1$, we have for $(u,v)\in \dot B_{2,1}^{\frac{d}{2}}\cap  \dot B_{2,1}^{\frac{d}{2}+1}$ that $uv\in  \dot B_{2,1}^{\frac{d}{2}+1}$ and the following inequality : $$\|uv\|_{ \dot B_{2,1}^{\frac{d}{2}+1}}\lesssim \|u\|_{ \dot B_{2,1}^{\frac{d}{2}}}\|v\|_{ \dot B_{2,1}^{\frac{d}{2}+1}}+\|u\|_{ \dot B_{2,1}^{\frac{d}{2}+1}}\|v\|_{ \dot B_{2,1}^{\frac{d}{2}}}.$$
\end{lemma}

The following lemma comes from the proof in \cite{BCD} of the following well-known property on Besov spaces : 
$$\|z(\alpha\cdot)\|_{\dot B_{2,1}^s}\simeq \alpha^{s-\frac{d}{2}}\|z\|_{\dot B_{2,1}^s} \quad \text{for all} \ \alpha>0.$$
\begin{lemma}\label{dilatation}
Let $s\in \R$ and $z\in\dot B_{2,1}^s.$ We have that :$$\|z(\alpha \ \cdot)\|_{\dot B_{2,1}^s}^{l^-,\varepsilon}\simeq \alpha^{s-\frac{d}{2}}\|z\|_{\dot B_{2,1}^s}^{l^-,\varepsilon\alpha^{-1}}; \quad \|z(\alpha \ \cdot)\|_{\dot B_{2,1}^s}^{l^+,\varepsilon}\simeq \alpha^{s-\frac{d}{2}}\|z\|_{\dot B_{2,1}^s}^{l^+,\varepsilon\alpha^{-1},\ \alpha^{-1}}; \quad \|z(\alpha \ \cdot)\|_{\dot B_{2,1}^s}^{h}\simeq \alpha^{s-\frac{d}{2}}\|z\|_{\dot B_{2,1}^s}^{h,\ \varepsilon\alpha^{-1}}$$ where we denote $\displaystyle \|\cdot \|_{\dot B_{2,1}^s}^{l^+,\ \alpha,\ \beta}\mathrel{\mathop:}=\sum\limits_{\underset{\alpha \leq 2^j \leq \frac{\beta}{4}}{j\in \Z}} 2^{js} \|\cdot \|_{L^2}$ and $\displaystyle \|\cdot \|_{\dot B_{2,1}^s}^{h,\alpha}\mathrel{\mathop:}=\sum\limits_{\underset{2^j \geq \alpha }{j\geq -2}} 2^{js} \|\cdot \|_{L^2}$. 
\end{lemma}

\begin{lemma} \label{fonction régulière besov}
Let $F:\R^d \rightarrow \R^p$ a smooth function with $F(0)=0$. Then for all $(p,r)\in [1,\infty]^2$, $s>0$ and $u\in \dot B_{p,r}^s\cap L^\infty$, we have $F(u)\in \dot B_{p,r}^s\cap L^\infty$ and $$\|F(u)\|_{\dot B_{p,r}^s}\leq C \|u\|_{\dot B_{p,r}^s},$$ with $C$ depending only on $\|u\|_{L^\infty}$, $F'$ (and higher order derivatives), $s$, $p$ and $d$.
\end{lemma}

\begin{lemma} \label{fonction regulière besov 2}
Let $g\in C^{\infty}(\R)$ such that $g'(0)=0$. Then, for all $u,v\in \dot B_{2,1}^s\cap L^\infty$ with $s>0$, $$\|g(v)-g(u)\|_{\dot B_{2,1}^s}\leq C\left(\|v-u\|_{L^\infty}\|(u,v)\|_{\dot B_{2,1}^s}+\|v-u\|_{\dot B_{2,1}^s} \|(u,v)\|_{L^\infty} \right)\cdotp$$
\end{lemma}

\end{document}